\documentclass[12pt]{amsart}

\newcommand{\indentalign}{\hspace{0.3in}&\hspace{-0.3in}}
\newcommand{\la}{\langle}
\newcommand{\ra}{\rangle}

\renewcommand{\Im}{\operatorname{Im}}
\newcommand{\sech}{\operatorname{sech}}
\newcommand{\defeq}{\stackrel{\rm{def}}{=}}

\newcommand{\supp}{\operatorname{supp}}

\newcommand{\core}{\textnormal{core}}
\newcommand{\rad}{\textnormal{rad}}
\newcommand{\sgn}{\operatorname{sgn}}

\newcommand{\lo}{\textnormal{lo}}
\newcommand{\med}{\textnormal{med}}
\newcommand{\hi}{\textnormal{hi}}

\usepackage{amssymb}
\usepackage{amsthm}
\usepackage{amsxtra}
\usepackage{graphicx}
\usepackage{color}

\newtheorem{theorem}{Theorem}[section]

\newtheorem{proposition}{Proposition}[section]
\newtheorem{lemma}[proposition]{Lemma}

\theoremstyle{remark}
\newtheorem{remark}[proposition]{Remark}

\numberwithin{equation}{section}
\ifx\pdfoutput\undefined
  \DeclareGraphicsExtensions{.pstex, .eps}
\else
  \ifx\pdfoutput\relax
    \DeclareGraphicsExtensions{.pstex, .eps}
  \else
    \ifnum\pdfoutput>0
      \DeclareGraphicsExtensions{.pdf}
    \else
      \DeclareGraphicsExtensions{.pstex, .eps}
    \fi
  \fi
\fi

\title[$H^1$ blow-ups on a sphere for 3d quintic NLS]
{Blow-up solutions on a sphere for the 3d quintic NLS in the energy space}
\author{Justin Holmer}
\address{Brown University}
\author{Svetlana Roudenko}
\address{Arizona State University}

\begin{document}

%[For private correspondence only]

\maketitle

\begin{abstract}
We prove that if $u(t)$ is a log-log blow-up solution, of the type studied by Merle-Rapha\"el \cite{MR}, to the $L^2$ critical focusing NLS equation $i\partial_t u +\Delta u + |u|^{4/d} u=0$ with initial data $u_0\in H^1(\mathbb{R}^d)$ in the cases $d=1, 2$, then $u(t)$ remains bounded in $H^1$ away from the blow-up point.  This is obtained without assuming that the initial data $u_0$ has any regularity beyond $H^1(\mathbb{R}^d)$.  As an application of the $d=1$ result, we construct an open subset of initial data in the radial energy space $H^1_{\rad}(\mathbb{R}^3)$ with corresponding solutions that blow-up on a sphere at positive radius for the 3d quintic ($\dot H^1$-critical) focusing NLS equation $i\partial_tu + \Delta u + |u|^4u=0$.   This improves Rapha\"el-Szeftel \cite{RS}, where an open subset in $H^3_{\rad}(\mathbb{R}^3)$ is obtained.
The method of proof can be summarized as follows: on the whole space, high frequencies above the blow-up scale are controlled by the bilinear Strichartz estimates.  On the other hand, outside the blow-up core, low frequencies are controlled by finite speed of propagation.
\end{abstract}

\section{Introduction}
\label{S:intro}

Consider the $L^2$ critical focusing nonlinear Schr\"odinger equation (NLS)
\begin{equation}
\label{E:NLS}
i\partial_t u + \Delta u + |u|^{4/d}u =0 \,,
\end{equation}
where $u=u(x,t)\in \mathbb{C}$ and $x\in \mathbb{R}^d$, in dimensions $d=1$ and $d=2$.  It is locally well-posed in $H^1(\mathbb{R}^d)$  and its solutions satisfy conservation of mass $M(u)$, momentum $P(u)$, and energy $E(u)$:
\begin{equation}
\label{E:cons}
M(u) = \|u\|_{L^2}^2 \,,  \quad P(u) = \Im \int \bar u \, \nabla u \, dx \,, \quad E(u) = \frac12 \|\nabla u \|_{L^2}^2 - \frac1{\frac{4}{d}+2} \|u\|_{L^{\frac{4}{d}+2}}^{\frac{4}{d}+2} \,,
\end{equation}
--see Tao \cite[Chap. 3]{Tao-book} and Cazenave \cite[Chap. 4]{Caz-book} for exposition and references.
The Galilean identity (see \cite[Exercise 2.5]{Tao-book}) transforms any solution to one with zero momentum, so there is no loss in considering only solutions $u(t)$ such that $P(u)=0$.

The unique (up to translation) minimal mass $H^1$ solution of
\begin{equation}
\label{E:ground-state}
-Q+ \Delta Q + |Q|^{4/d}Q=0 \,, \qquad Q=Q(x)
\end{equation}
is called the \emph{ground-state}.  It is smooth, radial, real-valued and positive, and exponentially decaying (see Tao \cite[Apx. B]{Tao-book}).  In the case $d=1$, we have explicitly
\begin{equation}
\label{E:ground-state-1d}
Q(x) = 3^{1/4}\sech^{1/2}(x) \,.
\end{equation}
Weinstein \cite{Wei} proved that solutions to \eqref{E:NLS} with $M(u)<M(Q)$ necessarily satisfy $E(u)>0$ and remain globally-in-time bounded in $H^1$ (do not blow-up in finite time).

Building upon the earlier heuristic and numerical result of Landman--Papanicolaou--Sulem--Sulem \cite{LPSS} and the first analytical result of Perelman \cite{Per},  Merle and Rapha\"el in a series of papers (see \cite{MR} and references therein) studied $H^1$ solutions to \eqref{E:NLS} such that
\begin{equation}
\label{E:MR-class}
E(u)<0 \,, \quad P(u)=0, \, \quad M(Q) < M(u) < M(Q) + \alpha^* \,,
\end{equation}
for some small absolute constant $\alpha^*>0$.  They showed that any such solution blows-up in finite time at the \emph{log-log rate} -- more precisely, they proved that there exists a \emph{threshold time} $T_0(u_0)>0$ and \emph{blow-up time} $T(u_0)>T_0(u_0)$ such that
\begin{equation}
\label{E:log-log-rate}
\|\nabla u (t)\|_{L_x^2} \sim \left( \frac{\log|\log(T-t)|}{T-t} \right)^{1/2} \,, \quad \text{for }T_0\leq t < T\,,
\end{equation}
where the implicit constant in \eqref{E:log-log-rate} is universal.
Moreover, if we take scale parameter $\lambda(t) = \|\nabla Q\|_{L^2}/\|\nabla u(t)\|_{L^2}$, then there exist parameters of position $x(t)\in \mathbb{R}^d$ and phase $\gamma(t) \in \mathbb{R}$ such that if we define the \emph{blow-up core}
\begin{equation}
\label{E:core}
u_{\core}(x,t) = \frac{e^{i\gamma(t)}}{\lambda(t)^{d/2}} Q\left( \frac{x-x(t)}{\lambda(t)} \right)  \,,
\end{equation}
and \emph{remainder} $\tilde u = u - u_{\core}$, then $\|\tilde u \|_{L^2} \leq \alpha_*$ and
\begin{equation}
\label{E:remainder-bounds}
\|\nabla \tilde u(t) \|_{L^2} \lesssim \left( \frac{1}{|\log(T-t)|^C(T-t)} \right)^{1/2}
\end{equation}
for some $C>1$.  There is, in addition, a well-defined \emph{blow-up point} $x_0 \defeq \lim_{t\nearrow T}x(t)$.  We refer to the region of space $\{ \, x\in \mathbb{R}^d \, | \, |x-x_0|>R \,\}$, for any fixed $R>0$, as the \emph{external} region.  While the Merle-Rapha\"el analysis accurately describes the activity of the solution in the blow-up core, the only information it directly yields about the external region is the bound \eqref{E:remainder-bounds}.

However, it is a consequence of the analysis in Rapha\"el \cite{Rap} that in the case $d=1$, $H^1$ solutions in the class \eqref{E:MR-class} have bounded $H^{1/2}$ norm in the external region all the way up to the blow-up time $T$.  In Holmer-Roudenko \cite{HR}, we extended this result to the case $d=2$.  Rapha\"el-Szeftel \cite{RS} established for $d=1$ that solutions with regularity $H^N$ for $N\geq 3$ satisfying \eqref{E:MR-class} remain bounded in the $H^{(N-1)/2}$-norm in the external region, and Zwiers \cite{Zwiers} extended this result to the case $d=2$.
These results leave open the possibility that there is a loss of roughly half the regularity
 in passing from the initial data to the solution in the external region at blow-up time.  The first main result of this paper is that such a loss \emph{does not occur}.  Specifically, we prove that $H^1$ solutions in the class \eqref{E:MR-class} remain bounded in the $H^1$-norm in the external region all the way up to the blow-up time, resolving an open problem posed in Rapha\"el-Szeftel \cite{RS} (Comment 1 on p. 976).

\begin{theorem}
\label{T:main}
Consider dimension $d=1$ or $d=2$.
Suppose that $u(t)$ is an $H^1$ solution to \eqref{E:NLS} in the Merle-Rapha\"el class \eqref{E:MR-class} (no higher regularity is assumed).  Let $T>0$ be the blow-up time and $x_0\in \mathbb{R}^d$ the blow-up point.  Then for any $R>0$,
$$\| \nabla u(t) \|_{L_{[0,T]}^\infty L^2_{|x-x_0|\geq R}} \leq C \,,$$
where $C$ depends\footnote{We did not see in the Merle-Rapha\"el papers the threshold time $T_0(u_0)$ or the blow-up time $T(u_0)$ estimated quantitatively in terms of properties of the initial data ($\|\nabla u_0\|_{L^2}$, $E(u_0)$, etc.).  If such dependence could be quantified, then the constant $C$ in Theorem \ref{T:main} could be quantified.}
 on $R$, $T_0(u_0)$, and $\|\nabla u_0\|_{L^2}$.
\end{theorem}

We remark that $H^1$, the \emph{energy space}, is a natural space in which to study the equation \eqref{E:NLS} since the conservation laws \eqref{E:cons} are defined and Lyapunov-Hamiltonian type methods, such as those used by Merle-Rapha\"el in their blow-up theory, naturally yield coercivity on $H^1$ quantities.   

The retention of regularity in the external region has applications to the construction of new blow-up solutions, with special geometry, for $L^2$ supercritical NLS equations.  Using their partial regularity methods,  Rapha\"el \cite{Rap} and Rapha\"el-Szeftel \cite{RS}
constructed spherically symmetric finite-time blow-up solutions to the quintic NLS
\begin{equation}
\label{E:NLS-5}
i\partial_t u + \Delta u + |u|^4u=0
\end{equation}
in dimension $d\geq 2$ that contract toward a sphere $|x|=r_0\sim 1$ following the one-dimensional quintic blow-up dynamics \eqref{E:log-log-rate}\eqref{E:core} in the radial variable near $r=r_0$.  Specifically, they showed there exists an open subset of initial data in some radial function class with corresponding solutions adhering to the above--described blow--up dynamics.  In \cite{Rap}, for $d=2$, an open subset of initial data in the radial energy space $H^1_{\rad}(\mathbb{R}^2)$ was obtained.
For $d=3$, in which case \eqref{E:NLS-5} is $\dot H^1$ critical, \cite{RS} obtained an open subset of initial data in a comparably ``thin'' subset $H^3_{\rad}(\mathbb{R}^3)$ of the radial energy space $H^1_\rad(\mathbb{R}^3)$.

As an application of the techniques used to prove Theorem \ref{T:main}, we prove, for $d=3$, the existence of an open subset of initial data in the full radial energy space $H^1_{\rad}(\mathbb{R}^3)$.  For the statement, take $Q$ to be the solution to \eqref{E:ground-state} in the case $d=1$, explicitly given by \eqref{E:ground-state-1d}.  The following theorem follows the motif of the $d=3$ case of Theorem 1 in \cite{RS} except that $\mathcal{P}$, the initial data, is an open subset of $H_\rad^1(\mathbb{R}^3)$ rather than $H^3_\rad(\mathbb{R}^3)$.

\begin{theorem}
\label{T:main-2}
There exists an open subset $\mathcal{P}\subset H^1_{\rad}(\mathbb{R}^3)$ such that the following holds true.  Let $u_0\in \mathcal{P}$ and let $u(t)$ denote the corresponding solution to \eqref{E:NLS-5} in the case $d=3$.  Then there exist a blow-up time $0<T<+\infty$ and parameters of scale $\lambda(t)>0$, radial position $r(t)>0$, and phase $\gamma(t)\in \mathbb{R}$ such that if we take
$$
u_\core (t,r) \defeq \frac{1}{\lambda(t)^{1/2}} Q\left( \frac{r-r(t)}{\lambda(t)} \right) e^{i\gamma(t)}
$$
and the remainder $\tilde u(t) \defeq u(t)-u_\core(t)$, then the following hold
\begin{enumerate}
\item The remainder converges in $L^2$:  $\tilde u (t) \to u^*$ in $L^2(\mathbb{R}^3)$ as $t\nearrow T$.
\item The position of the singular sphere converges: $r(t) \to r_0>0$ as $t\nearrow T$.
\item The solution contracts toward the sphere at the log--log rate:
$$\lambda(t) \left( \frac{\log|\log(T-t)|}{T-t} \right)^{1/2} \to \frac{\sqrt{2\pi}}{\|Q\|_{L^2}} \text{ as }t\nearrow T\,.$$
\item The solution remains $H^1$-small away from the singular sphere:  For each $R>0$, $\|u(t)\|_{H^1_{|r-r(T)|\geq R}(\mathbb{R}^3)} \leq \epsilon$.
\end{enumerate}
\end{theorem}

The 3d quintic NLS equation \eqref{E:NLS-5} is energy-critical, and the global well-posedness and scattering problem is one of several critical regularity problems that has received a lot of attention in the last decade \cite{Bou-2,CKSTT,KM}.  The global well-posedness for small data in $\dot H^1$ is classical and follows from the Strichartz estimates.  Our Theorem \ref{T:main-2} takes a large, but special ``prefabricated'' approximate blow-up solution, and installs it near radius $r=1$ on top of a small global $H^1$ background. The main difficulty, of course, is showing that the two different components -- the blow-up portion on the one hand, and the evolution of the small $\dot H^1$ background on the other, have limited interaction and can effectively evolve separately.  Thus, it is not surprising that the techniques to prove Theorem \ref{T:main} are relevant to this analysis.

We now outline the method used to prove Theorem \ref{T:main}.  We start with a given blow-up solution $u(t)$ in the Merle-Rapha\"el class, and by scaling and shifting this solution, it suffices to assume that the blow-up point is $x_0=0$ and the blow-up time is $T=1$, and moreover, \eqref{E:log-log-rate} holds over times $0\leq t< 1$.   Since \eqref{E:NLS} is $L^2$ critical, the size of the $L^2$ norm is highly relevant.  By mass conservation, we know that $\|P_N u(t)\|_{L_x^2} \lesssim 1$ for all $N$ and all $0\leq t<1$, where $P_N$ denotes the Littlewood-Paley frequency projection.  However, \eqref{E:log-log-rate} shows that for $N \gg (1-t)^{-(1+\delta)/2}$, we have $\|P_N u (t) \|_{L_x^2} \lesssim N^{-1}(1-t)^{-(1+\delta)/2}$, which is a better estimate for these large frequencies $N$.  In \S 3, we show that this smallness of high frequencies reinforces itself and ultimately proves that for $N\gg (1-t)^{-(1+\delta)/2}$, the solution is $H^1$ bounded.  This is achieved using dispersive estimates typically employed in local well-posedness arguments -- the Strichartz and Bourgain's bilinear Strichartz estimates -- after the equation has been restricted to high frequencies.
We note that this improvement of regularity at high frequencies is proved \emph{globally in space}.

For the Schr\"odinger equation, frequencies of size $N$ propagate at speed $N$, and thus, travel a distance $O(1)$ over a time $N^{-1}$.  Therefore, at time $t<1$, a component of the solution in the blow-up core at frequency $N$ will effectively only make it out of the blow-up core and into the external region before the blow-up time provided  $N\gtrsim (1-t)^{-1}$.  Thus, we expect that the blow-up action, which is taking place at frequency $\sim (1-t)^{-1/2}\log|\log(1-t)| \ll (1-t)^{-1}$, will not be able to exit the blow-up core before blow-up time.  This is the philosophy behind the analysis in \S \ref{S:finite-speed}.  Recall that in \S \ref{S:high-frequency}, we have controlled the solution at frequencies above $(1-t)^{-(1+\delta)/2}$.  In \S \ref{S:finite-speed}, we apply a spatial localization to the external region, and then look to control the remaining low frequencies, i.e., those frequencies below $(1-t)^{-(1+\delta)/2}$.  We examine the equation solved by $P_{\leq (1-t)^{-3/4}} \psi u(t)$, where $\psi$ is a spatial restriction to the external region. In estimating the inhomogeneous terms, we can make use of the frequency restriction to exchange $\alpha$-spatial derivatives for a time factor $(1-t)^{-3\alpha /4}$.  This enables us to prove a low-frequency recurrence:  the $H^s$ size of the solution in the external region is bounded by the $H^{s-\frac18}$ size of the solution in a slightly larger external region.  Iteration gives the $H^1$ boundedness.

The structure of the paper is as follows.
Preliminaries on the Strichartz and bilinear Strichartz estimates appear in \S\ref{S:estimates}.  The proof of Theorem \ref{T:main} is carried out in \S\ref{S:high-frequency}-\ref{S:finite-speed}.  The proof of Theorem \ref{T:main-2} is carried out in \S\ref{S:3d-sphere}.

\subsection{Acknowledgements}
J.H. is partially supported by a Sloan fellowship and NSF grant DMS-0901582.
S.R. is partially supported by NSF grant DMS-0808081. J.H. thanks Mike Christ and Daniel Tataru for patient mentorship, in work related to the paper \cite{KT}, on the use of the bilinear Strichartz estimates.

\section{Standard estimates}
\label{S:estimates}

All of the estimates outlined in this section are now classical and well-known.  Let $P_N$, $P_{\leq N}$, $P_{\geq N}$ denote the Littlewood-Paley frequency projections.

We say that $(q,p)$ is an \emph{admissible} pair if $2\leq p \leq \infty$ and
$$
\frac{2}{q}+\frac{d}{p}=\frac{d}{2},
$$
excluding the case $d=2$, $q=2$, $p=\infty$.

\begin{lemma}[Strichartz estimate]
\label{L:Strichartz}
If $(q,p)$ is an admissible pair, then
$$
\| e^{it\Delta} \phi \|_{L_t^qL_x^p} \lesssim \|\phi\|_{L_x^2}.
$$
\end{lemma}

\begin{proof}
See Strichartz \cite{Str} and Keel-Tao \cite{Keel-Tao}.
\end{proof}

\begin{lemma}[Bourgain bilinear Strichartz estimate]
\label{L:bilinear-Strichartz}
Suppose that $N_1 \ll N_2$.  Then
\begin{equation}
\label{E:Bourgain1}
\| P_{N_1} e^{it\Delta}\phi_1 \; P_{N_2}e^{it\Delta}\phi_2 \|_{L_t^2L_x^2} \lesssim \left(\frac{N_1^{d-1}}{N_2}\right)^{1/2} \|\phi_1\|_{L_x^2}\|\phi_2\|_{L_x^2},
\end{equation}
\begin{equation}
\label{E:Bourgain2}
\| P_{N_1} e^{it\Delta}\phi_1 \; \overline{P_{N_2}e^{it\Delta}\phi_2} \|_{L_t^2L_x^2} \lesssim \left(\frac{N_1^{d-1}}{N_2}\right)^{1/2} \|\phi_1\|_{L_x^2}\|\phi_2\|_{L_x^2}.
\end{equation}
\end{lemma}

\begin{proof}
For the $2d$ estimate \eqref{E:Bourgain1} see Bourgain \cite{Bou} Lemma 111; %for the $d=2$ case;
the $1d$ case appears in \cite{CKSTT01} Lemma 7.1; another nice proof is given in Prop. 3.5 in  Koch-Tataru \cite{KT}, the other dimensions are analogous. We review the $1d$ proof to show that the second estimate \eqref{E:Bourgain2} holds as well.

Denote $u = e^{it \Delta} (P_{N_1} \phi_1)$ and $v = e^{\pm i t \Delta} (P_{N_2} \phi_2)$.
Then in the $1d$ case,
\begin{align}
\label{E:B1}
\widehat{u v} (\xi, \tau) &= \int_{\xi_1+\xi_2 = \xi} \widehat{P_{N_1} \phi_1}(\xi_1)
\widehat{P_{N_2} \phi_2}(\xi_2) \, \delta(\tau - (\xi_1^2 \pm \xi_2^2)) \, d\xi_1 \\
& = \frac1{|g^\prime_{\xi_1}(\xi_1,\xi_2)|} \, \widehat{P_{N_1}\phi_1}\widehat{P_{N_2}\phi_2}_{|_{(\xi_1,\xi_2)}},
\label{E:B2}
\end{align}
where $g(\xi_1,\xi_2) = \tau-(\xi_1^2\pm \xi_2^2)$, thus, $|g'_{\xi_1}| = 2 |\xi_1\pm\xi_2|$.
To estimate the $L^2_{\xi,\tau}$ norm of $uv$, we square the expression above and integrate in $\tau$ and $\xi$. Changing variables $(\tau, \xi)$ to $(\xi_1,\xi_2)$ with
$\tau = \xi_1^2 \pm \xi_2^2$ and $\xi = \xi_1+\xi_2$, we obtain
$d\tau d\xi = J \, d\xi_1 d\xi_2$ with the Jacobian $J=2|\xi_1 \pm \xi_2|$ which is of size $N_2$ (note that $\pm$ does not matter here, since $N_2 \gg N_1$).
Bringing the square inside, %and switching to the original variables $(\xi_1, \xi_2)$,
we get
$$
\|u v\|_{L^2_x}^2 \lesssim \int_{|\xi_1|\sim N_1, |\xi_2|\sim N_2} |\widehat{\phi_1}(\xi_1)|^2 |\widehat{\phi_2}(\xi_2)|^2 \, \frac{d\xi_1 \, d\xi_2}{|\xi_1 \pm \xi_2|} \lesssim \frac1{N_2} \|\phi_1\|_{L^2_x}^2 \|\phi_2\|_{L^2_x}^2.
$$
%Note that in $2d$ and higher dimensions, one needs to measure the size of the set $|\xi_1| \sim N_1$ when estimating \eqref{E:B2} from \eqref{E:B1}, thus, obtaining the factor $N_1^{d-1}$.
\end{proof}

Now we introduce the Fourier restriction norms.  For $\tilde u \in \mathcal{S}(\mathbb{R}^{1+d})$
\begin{align*}\|\tilde u\|_{X_{s,b}} &= \|\la D_t\ra^b \la D_x\ra^s e^{-it\Delta}\tilde u(\cdot,t)\|_{L_t^2L_x^2} \\
&= \left( \int_\xi \int_\tau |\widehat{\tilde u}(\xi,\tau)|^2 \la \xi\ra^{2s} \la \tau+|\xi|^2\ra^{2b} \, d\xi\, d\tau \right)^{1/2}.
\end{align*}
If $I\subset \mathbb{R}$ is an open subinterval and $u\in \mathcal{D}'(I\times \mathbb{R}^d)$, define
$$
\|u\|_{X_{s,b}(I)} = \inf_{\tilde u} \|\tilde u \|_{X_{s,b}},
$$
where the infimum is taken over all distributions $\tilde u \in \mathcal{S}'(\mathbb{R}^{1+d})$ such that $\tilde u\big|_I = u$.

\begin{lemma}
If $\theta$ is a function such that $\supp\theta \subset I$, then for all $0<b<1$,
\begin{equation}
\label{E:step-150}
\| \theta u \|_{X_{s,b}} \lesssim (\|\theta\|_{L^\infty}+\|D_t^{\max(\frac12,b)} \theta\|_{L^2})\|u\|_{X_{s,b}(I)} \,.
\end{equation}
If $0\leq b<\frac12$ and $\chi_I$ is the (sharp) characteristic function of the time interval $I$, then
\begin{equation}
\label{E:step-151}
\| \chi_I u \|_{X_{s,b}} \sim \|u\|_{X_{s,b}(I)} \,.
\end{equation}
\end{lemma}
\begin{proof}
It suffices to take $s=0$.
The inequality \eqref{E:step-150} follows from the fractional Leibniz rule.  To address \eqref{E:step-151}, we note that Jerison-Kenig \cite{JK95} prove that for $-\frac12<b<\frac12$, $\|\chi_{(0,+\infty)} f \|_{H_t^b} \lesssim \|f\|_{H_t^b}$.  Consequently, $\|\chi_I f \|_{H_t^b} \lesssim \|f\|_{H_t^b}$ for any time interval $I$.  Let $\tilde u$ be an extension of $u$ (meaning $\tilde u\big|_{I} = u$) so that $\|\tilde u \|_{X_{0,b}} \leq 2\|u\|_{X_{0,b}(I)}$. Then
\begin{align*}
\|\chi_I u \|_{X_{0,b}} &= \| \la D_t \ra^b e^{-it\Delta}\chi_I \tilde u \|_{L_t^2L_x^2} \\
&=  \| \; \| \chi_I e^{-it\Delta}\tilde u \|_{H_t^b} \; \|_{L_x^2} \\
&\lesssim \| \; \| e^{-it\Delta}\tilde u \|_{H_t^b} \; \|_{L_x^2} \\
& = \|\tilde u \|_{X_{0,b}} \\
&\leq 2 \|u\|_{X_{0,b}(I)}\,.
\end{align*}
On the other hand, the inequality $\|u\|_{X_{0,b}(I)} \lesssim \|\chi_I u\|_{X_{0,b}}$ is trivial, since $\chi_I u $ is an extension of $u\big|_{I}$.
\end{proof}

\begin{lemma}
\label{L:flow}
If $i\partial_t u + \Delta u = f$  on a time interval $I=(a,d)$ with $|I|=O(1)$, then
\begin{enumerate}
\item For $\frac12<b\leq 1$,  taking $I'=(a-\omega,d+\omega)$, $0<\omega\leq 1$, we have
\begin{equation}
\label{E:step-101}
\|u(t) - e^{i(t-a)\Delta}u(a)\|_{X_{0,b}(I)} \lesssim \omega^{\frac12-b} \|f\|_{X_{0,b-1}(I')}.
\end{equation}
\item For $0\leq b<\frac12$,
\begin{equation}
\label{E:step-102}
\|u(t) - e^{i(t-a)\Delta}u(a)\|_{X_{0,b}(I)} \lesssim \|f\|_{L_I^1L_x^2} \,.
\end{equation}
\end{enumerate}
Moreover, for all $b$,
$$
\| e^{i(t-a)\Delta}\phi \|_{X_{0,b}(I)} \lesssim \|\phi\|_{L_x^2} \,.
$$
\end{lemma}
\begin{proof}
Without loss, we take $a=0$.  First we consider \eqref{E:step-101}.  Since, for $t\in I$,
$$
e^{-it\Delta}u(\cdot,t) = u(0) -i \theta(t)\int_0^t e^{-it'\Delta} \theta(t') f(\cdot,t') \, dt'\,,
$$
where $\theta$ is a cutoff function such that $\theta(t)=1$ on $I$ and $\supp \theta \subset I'$,
the estimate reduces to the space-independent estimate
\begin{equation}
\label{E:integral}
\left\| \theta(t) \int_0^t h(t') \, dt' \right\|_{H_t^b} \lesssim \|h\|_{H_t^{b-1}}, \qquad \text{for }\tfrac12<b\leq 1
\end{equation}
by \eqref{E:step-150}.
Now we prove estimate \eqref{E:integral}. Divide $h=P_{\leq 1}h + P_{\geq 1}h$ and use that $\int_0^t P_{\geq 1} h(t') = \frac{1}{2}\int (\sgn(t-t') + \sgn(t')) P_{\geq 1} h(t') \,dt'$ to obtain the decomposition
$$
\theta(t) \int_0^t h(t')\,dt' = H_1(t)+H_2(t)+H_3(t),
$$
where
\begin{align*}
&H_1(t) = \theta(t) \int_0^t P_{\leq 1}h(t') \, dt' \\
&H_2(t) = \tfrac12 \theta(t) [\sgn* P_{\geq 1} h](t) \, dt' \\
&H_3(t) = \tfrac12 \theta(t) \int_{-\infty}^{+\infty} \sgn(t') P_{\geq 1}h(t') \,d t'. \end{align*}
We begin by addressing term $H_1$.
By Sobolev embedding (recall $\frac12<b\leq 1$) and the $L^p \to L^p$ boundedness of the Hilbert transform for $1<p<\infty$,
$$
\| H_1 \|_{H_t^b} \lesssim \|H_1\|_{L_t^2} + \| \partial_tH_1 \|_{L_t^{2/(3-2b)}}\,.
$$
Using that $|I|=O(1)$ and $\|P_{\leq 1} h \|_{L_t^\infty} \lesssim \|h\|_{H_t^{b-1}}$, we thus conclude
$$
\| H_1 \|_{H_t^b} \lesssim (\| \theta \|_{L_t^2} + \|\theta\|_{L_t^{2/(3-2b)}} + \|\theta'\|_{L_t^{2/3-2b}}) \|h\|_{H_t^{b-1}} \,.
$$

Next we address the term $H_2$.  By the fractional Leibniz rule,
$$
\|H_2\|_{H_t^b} \lesssim \|\la D_t \ra^b \theta \|_{L_t^2}\| \sgn * P_{\geq 1}h \|_{L_t^\infty} + \|\theta\|_{L_t^\infty} \| \la D_t \ra^b (\sgn*P_{\geq 1}h) \|_{L_t^2} \,.
$$
However,
$$
\| \sgn * P_{\geq 1}h \|_{L_t^\infty} \lesssim \| \la \tau \ra^{-1} \hat h(\tau) \|_{L_\tau^1} \lesssim \| h \|_{H_t^{b-1}} \,.
$$
On the other hand,
$$
\| \la D_t \ra^b \sgn * P_{\geq 1}h \|_{L_t^2}  \lesssim \| \la \tau\ra^b \la \tau \ra^{-1} \hat h(\tau) \|_{L_\tau^2} \lesssim \|h\|_{H_t^{b-1}}\,.
$$
Consequently,
$$
\| H_2 \|_{H_t^b} \lesssim ( \| \la D_t\ra^b \theta \|_{L_t^2} + \|\theta\|_{L_t^\infty}) \|h\|_{H_t^{b-1}} \,.
$$
For term $H_3$, we have
$$
\|H_3 \|_{H_t^b} \lesssim \|\theta \|_{H_t^b} \left\| \int_{-\infty}^{+\infty} \sgn(t') P_{\geq 1}h(t') \, dt' \right\|_{L_t^\infty} \,.
$$
However, the second term is handled via Parseval's identity
$$
\int_{t'} \sgn(t') P_{\geq 1}h(t') \,dt' = \int_{|\tau|\geq 1} \tau^{-1} \hat h(\tau) \, d\tau \,,
$$
from which the appropriate bounds follow again by Cauchy-Schwarz.    Collecting our estimates for $H_1$, $H_2$, and $H_3$, we have
$$
\left\| \theta(t) \int_0^t h(t') \, dt'  \right\|_{H_t^b} \lesssim C_\theta \|h\|_{H_t^{b-1}},
$$
where
\begin{align*}
C_\theta &= \|\theta\|_{L_t^2} + \|\theta' \|_{L_t^{2/(3-2b)}} + \| \la D_t \ra^b \theta\|_{L_t^2}+ \|\theta\|_{L_t^{2/(3-2b)}} + \|\theta \|_{L_t^\infty} \\
&\lesssim \omega^{\frac12-b}\,.
\end{align*}
This completes the proof of \eqref{E:step-101}.
Next, we prove \eqref{E:step-102}.   We have
$$
e^{-it\Delta}u(\cdot,t) = u(0) -i \int_0^t e^{-it'\Delta} f(\cdot,t') \, dt',
$$
and thus, \eqref{E:step-102} reduces, by \eqref{E:step-151}, to
\begin{equation}
\label{E:integral2}
\left\| \chi_I \int_0^t g(t') \, dt' \right\|_{H_t^b} \lesssim \|g\|_{L_I^1},  \qquad \text{for }0\leq b< \tfrac12 \,.
\end{equation}
To prove \eqref{E:integral2}, note that
$$ 
\chi_I(t) \int_0^t g(t') \, dt' = \chi_I(t) [\chi_I * (g\chi_I)](t) \,.
$$
Hence,
$$ 
\left\| \chi_I \int_0^t g(t') \, dt'  \right\|_{H_t^b} \lesssim \| \la D \ra^b \chi_I \|_{L_t^2} \|g\|_{L_I^1}\,.
$$
The Fourier transform of $\chi_I$ is smooth and decays like $|\tau|^{-1}$ as $|\tau|\to \infty$, and hence, $\| \la D \ra^b \chi_I \|_{L_t^2} <\infty$ for $0\leq b<\frac12$.
\end{proof}

\begin{lemma}[Strichartz estimate]
\label{L:X-Strichartz}
If $(q,r)$ is an admissible pair,
then we have the embedding
$$
\|u \|_{L_I^q L_x^p} \lesssim \|u\|_{X_{0,\frac12+\delta}(I)}.
$$
\end{lemma}
\begin{proof}
We reproduce the well-known argument.  Replace $u$ by an extension to $t\in \mathbb{R}$ so that $\|u\|_{X_{0,\frac12+\delta}} \leq 2 \|u \|_{X_{0,\frac12+\delta}(I)}$.  Write
$$
u(x,t) = \int_\xi \int_\tau  e^{it\tau} e^{ix\cdot \xi} \hat u(\xi,\tau) \, d\tau \, d\xi \,.
$$
Change variables $\tau \mapsto \tau-|\xi|^2$ and apply Fubini to obtain
$$
u(x,t) = \int_\tau e^{it\tau} \int_\xi  e^{-it |\xi|^2} e^{ix\cdot \xi} \hat u(\xi,\tau-|\xi|^2) \, d\xi \, d\tau \,.
$$
Define $f_\tau(x)$ by $\hat f_\tau(\xi) = \hat u(\xi,\tau-|\xi|^2)$.  Then the above reads
$$
u(x,t) = \int_\tau e^{it\tau} e^{it\Delta} f_\tau(x) \, d\tau \,,
$$
and hence,
$$
|u(x,t)| \leq \int_\tau |e^{it\Delta}f_\tau(x)| \, d\tau\,.
$$
Apply the Strichartz norm, the Minkowski integral inequality, appeal to Lemma \ref{L:Strichartz}, and invoke Plancherel to obtain
$$
\|u \|_{L_I^q L_x^p} \lesssim \int_\tau \|\hat f_\tau(\xi)\|_{L^2_\xi} \, d\tau.
$$
The argument is completed using Cauchy-Schwarz in $\tau$ (note that we need $b>\frac12$, since
$\int_\mathbb{R} \la\tau\ra^{-2b} \, d\tau$ has to be finite).
\end{proof}

\begin{lemma}[Bourgain bilinear Strichartz estimate]
\label{L:X-bilinear-Strichartz}
Let $N_1\ll N_2$. Then
$$
\| P_{N_1} u_1 \; P_{N_2}u_2 \|_{L_I^2 L_x^2} \lesssim \left( \frac{N_1^{d-1}}{N_2} \right)^{1/2} \|u_1\|_{X_{0,\frac12+\delta}(I)} \|u_2\|_{X_{0,\frac12+\delta}(I)},
$$
$$
\| P_{N_1} u_1 \; \overline{P_{N_2}u_2} \|_{L_I^2 L_x^2} \lesssim \left( \frac{N_1^{d-1}}{N_2} \right)^{1/2} \|u_1\|_{X_{0,\frac12+\delta}(I)} \|u_2\|_{X_{0,\frac12+\delta}(I)}.
$$
\end{lemma}
\begin{proof}
We reproduce the well-known argument.  As in the proof of Lemma \ref{L:X-Strichartz}, taking $f_{j,\tau}(x)$ defined by $\hat f_{j,\tau}(\xi) = \hat u_1(\xi,\tau-|\xi|^2)$, we have
$$
u_j(x,t) = \int_\tau e^{it\tau} \; e^{it\Delta}f_{j,\tau}(x) \, d\tau \,.
$$
Plug these into the expression $\| P_{N_1} u_1 \; P_{N_2}u_2 \|_{L_t^2 L_x^2}$, and then estimate using Lemma \ref{L:bilinear-Strichartz}.
\end{proof}

We need to take $b=\frac12-\delta$ in some places.  In those situations, we use

\begin{lemma}[interpolated Strichartz]
\label{L:interp-1}
Take $d=1$ or $d=2$ and suppose that $0\leq b<\frac12$ and $2\leq p\leq\infty$, $2<q\leq \infty$ satisfy
\begin{align}
\label{E:interp1}
&\frac{2}{q}+\frac{d}{p}>\frac{d}{2}+(1-2b) \\
\label{E:interp2}
&\frac{2}{q}-\frac{1}{p}\leq \frac{1}{2} & \text{in the case }d=1 \text{ only}
\end{align}
(see Fig. \ref{F:interpolation}).
Then
\begin{equation}
\label{E:interp3}
\|u\|_{L_I^q L_x^p} \lesssim \|u\|_{X_{0,b}(I)}.
\end{equation}
with implicit constant dependent upon the size of the gap from equality in \eqref{E:interp1}.
\end{lemma}
\begin{proof}
Let
\begin{equation}
\label{E:interp5}
\alpha \defeq \frac12\left( \frac{2}{q}+\frac{d}{p}-\frac{d}{2}-(1-2b)\right)>0 \,.
\end{equation}
Using $0\leq \theta\leq 1$ as an interpolation parameter,
we aim to deduce \eqref{E:interp3} by interpolation between
\begin{equation}
\label{E:interp4}
\|u\|_{L_t^{\tilde q}L_x^{\tilde p}} \lesssim \|u\|_{X_{0,\frac{b}{2(b-\alpha)}}} \,,
\end{equation}
with weight $\theta$, for some Strichartz admissible pair $(\tilde q,\tilde p)$, and the trivial estimate (equality, in fact)
\begin{equation}
\|u\|_{L_t^2L_x^2} \lesssim \|u\|_{X_{0,0}} \,,
\end{equation}
with weight $1-\theta$.  The interpolation conditions read
\begin{equation}
\label{E:interp6}
\begin{aligned}
&\frac{1}{q} = \frac{\theta}{\tilde q} + \frac{1-\theta}{2} \\
&\frac{1}{p} = \frac{\theta}{\tilde p} + \frac{1-\theta}{2} \,.
\end{aligned}
\end{equation}
Multiplying the first of these relations by $2$ and adding $d$ times the second, and using the Strichartz admissibility condition for $(\tilde q, \tilde p)$, we obtain
$$
\frac{2}{q}+\frac{d}{p} = \frac{d}{2}+(1-\theta) \,.
$$
Combining this relation with \eqref{E:interp5}, we obtain $\theta = 2b-2\alpha$.  We can then solve for $\tilde q$ and $\tilde p$ using \eqref{E:interp6}.
\end{proof}

\begin{figure}
  \caption{The enclosed triangular region gives the values of $(1/q,1/p)$ meeting the hypotheses of Lemma \ref{L:interp-1}.  The top frame is the case $d=1$ and the bottom frame is the case $d=2$.  The proof of Lemma \ref{L:interp-1} involves interpolating between a point on the line $\frac{2}{q}+\frac{d}{p}=\frac{d}2$ and the point $(\frac12,\frac12)$. }
  \centering
    \includegraphics[scale=0.75]{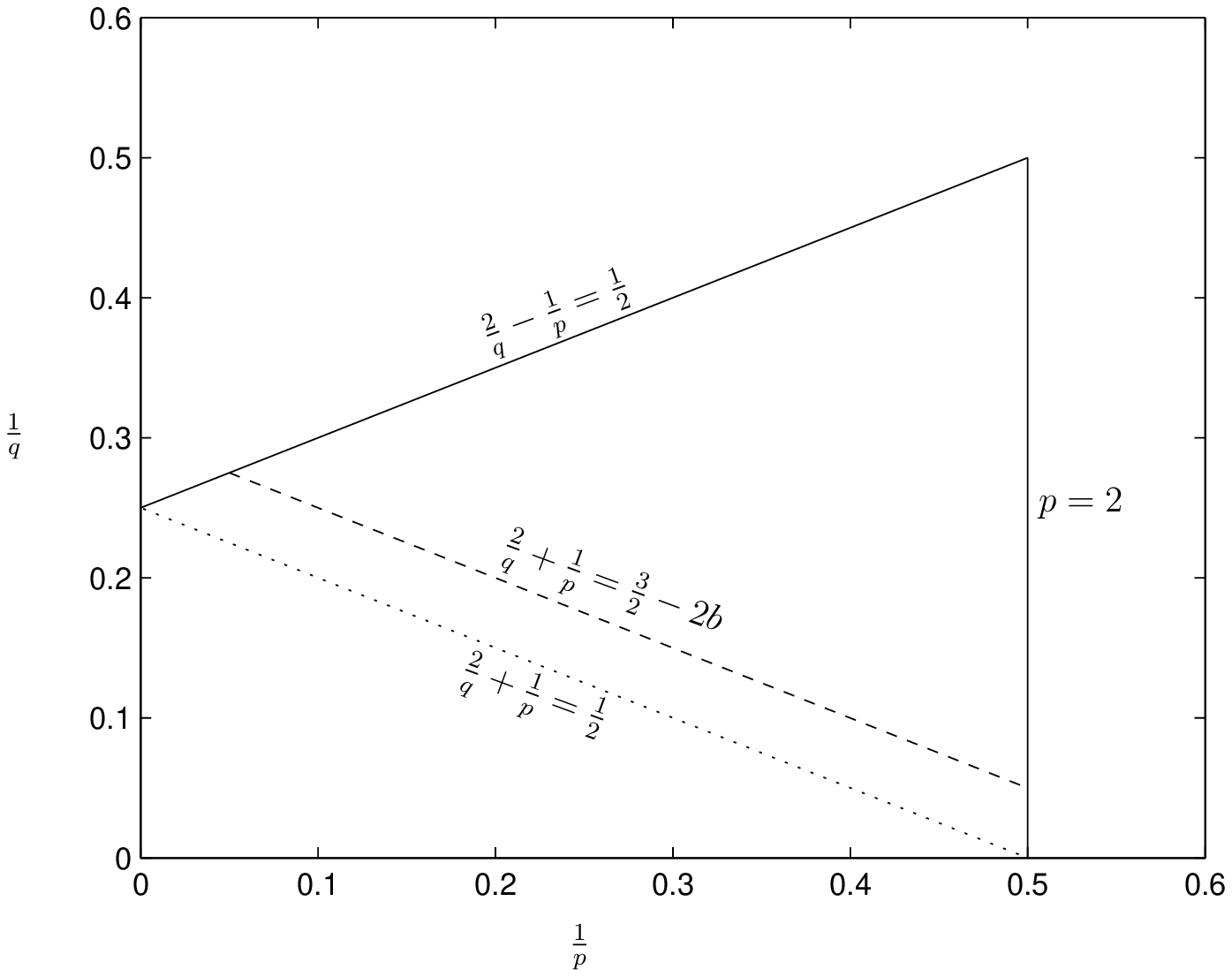}
\includegraphics[scale=0.75]{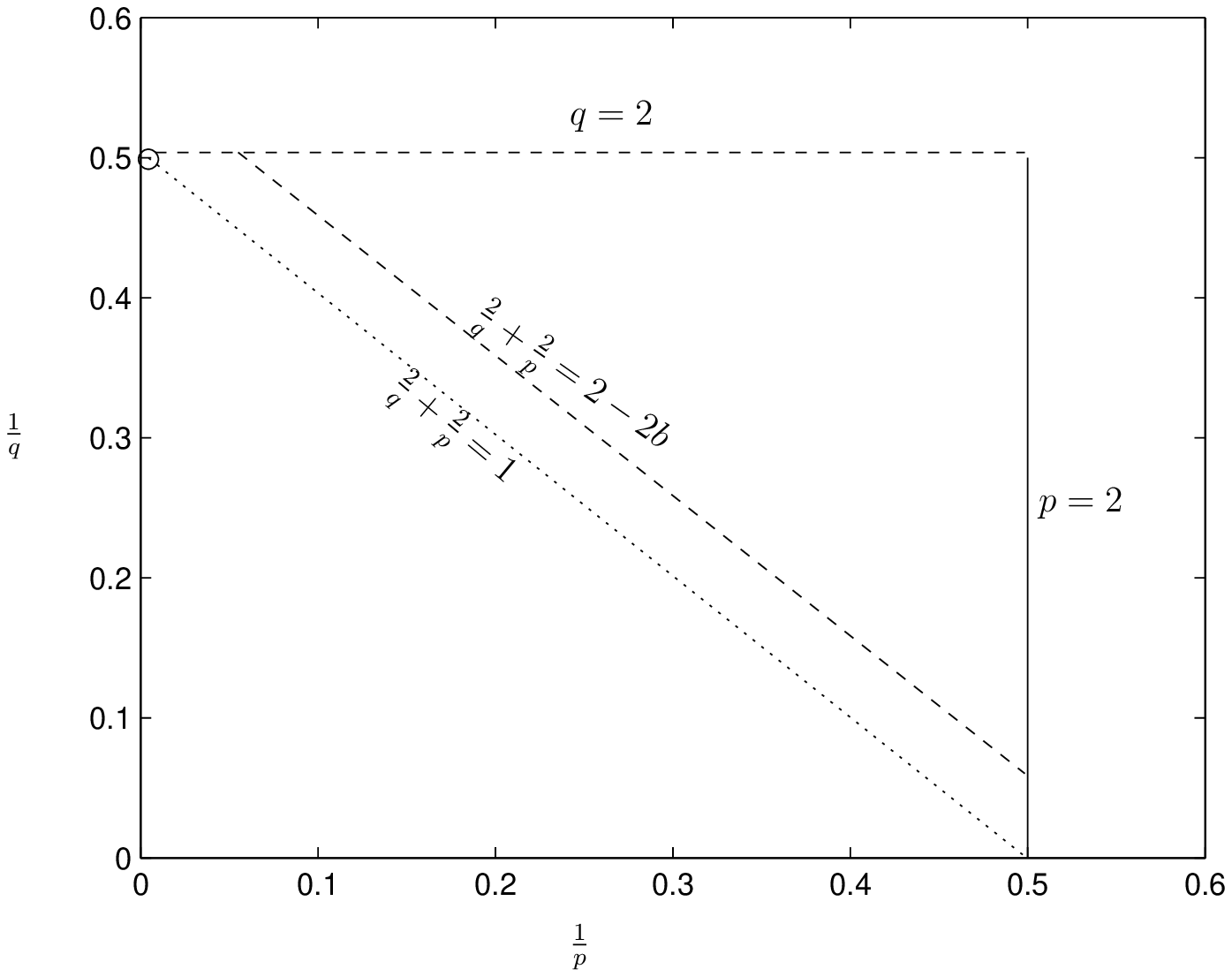}
\label{F:interpolation}
\end{figure}

\begin{lemma}[interpolated bilinear Strichartz]
\label{L:interp-2}
Let $d=1$ or $d=2$ and $N_1\ll  N_2$.
Then
$$
\|P_{N_1} u_1 \; P_{N_2}u_2 \|_{L_I^2 L_x^2} \lesssim \frac{N_1^{\frac12(d-1)}}{N_2^{\frac12-\delta'}} \|u_1\|_{X_{0,\frac12-\delta}(I)} \|u_2\|_{X_{0,\frac12-\delta}(I)}.
$$
\end{lemma}
\begin{proof} First, observe that
\begin{equation}
\label{E:step-132}
\|P_{N_1} u_1 \; P_{N_2}u_2 \|_{L_I^2 L_x^2} \lesssim \|u_1\|_{L_I^4L_x^4} \|u_2\|_{L_I^4L_x^4}.
\end{equation}
In the case $d=1$, $L_I^4L_x^4$ interpolates between $L_I^6 L_x^6$ and $L_I^2 L_x^2$, and thus, by Lemma \ref{L:interp-1}, $\|u_j\|_{L_I^4L_x^4} \lesssim \|u_j \|_{X_{0,\frac38+\delta}(I)}$.  We conclude that
$$
\|P_{N_1} u_1 \; P_{N_2}u_2 \|_{L_I^2 L_x^2} \lesssim  \|u_1 \|_{X_{0,\frac38+\delta}(I)} \|u_2 \|_{X_{0,\frac38+\delta}(I)}.
$$
Interpolating this with the result of Lemma \ref{L:X-bilinear-Strichartz} completes the proof in the case $d=1$.

In the case $d=2$,  we still begin with \eqref{E:step-132}.  Fix $\epsilon>0$ small.  By Sobolev embedding,
$$
\|P_{N_j} u_j \|_{L_I^4L_x^4} \lesssim N_j^\epsilon \|P_{N_j} u_j\|_{L_I^4L_x^\frac{4}{1+2\epsilon}} \,.
$$
By Lemma \ref{L:interp-1}, we have
$$
\|P_{N_j} u_j\|_{L_I^4L_x^\frac{4}{1+2\epsilon}} \lesssim \|u_j\|_{X_{0,\frac12(1-\epsilon)+}} \,.
$$
Plugging into \eqref{E:step-132}, we obtain
$$
\|P_{N_1} u_1 \; P_{N_2}u_2 \|_{L_I^2 L_x^2} \lesssim N_2^{2\epsilon} \|u_1\|_{X_{0,\frac12(1-\epsilon)+}} \|u_2\|_{X_{0,\frac12(1-\epsilon)+}} \,.
$$
Interpolating this with the result of Lemma \ref{L:X-bilinear-Strichartz} completes the proof in the case $d=2$.
\end{proof}

\begin{remark}
After this section we will adopt the following notation: instead of $X_{s,\frac12+\delta}$ we will simply write $X_{s, \frac12+}$. If an expression has two different Bourgain spaces, it will mean that the delta's will be different. Similarly, if an expression involves $\delta$ in the estimate on the right side, it will mean that this $\delta$ will be different from the one which would be chosen for spaces such as $X_{s,\frac12+}$ or $L^{p-}$.
\end{remark}

The following is a simple consequence of the pseudodifferential calculus -- see Stein \cite{Stein}, Chapter VI, \S 2, Theorem 1 on p. 234 and \S 3, Theorem 2 on p. 237; see also Evans-Zworski \cite{EZ}.

\begin{lemma}
\label{L:commutator}
Suppose that $\phi$ is a smooth function on $\mathbb{R}$ such that $\|\partial_x^\alpha \phi \|_{L^\infty} \leq c_\alpha$ for all $\alpha\geq 0$.  Then for $N\geq 1$,
$$
\|P_{\geq N}(\phi g) - \phi P_{\geq N} g \|_{L^2} \lesssim N^{-1} \|g\|_{L^2} \,.
$$
\end{lemma}
\begin{proof}
Let $\chi(\xi)$ be a smooth function that is $1$ for $|\xi|\geq 1$ and is $0$ for $|\xi|\leq \frac12$.  $P_{\geq N}$ is a pseudodifferential operator with symbol $\chi(N^{-1}\xi)$ and $M_\phi$, the operator of multiplication by $\phi$, is a pseudodifferential operator with symbol $\phi(x)$.   The commutator $[P_N,M_\phi]$ has symbol with top-order asymptotic term $N^{-1}\chi'(N^{-1}\xi) \phi'(x)$.  The result then follows from the $L^2\to L^2$ boundedness of $0$-order operators.
\end{proof}

\section{Additional high-frequency regularity}
\label{S:high-frequency}

In this section, we begin the proof of Theorem \ref{T:main} by showing improved regularity at high frequencies, above the blow-up scale, \emph{with no restriction in space} -- this appears as Prop. \ref{P:high-freq} below.  In \S\ref{S:finite-speed} below, we will complete the proof of Theorem \ref{T:main} by appealing to a finite-speed of propagation argument for lower frequencies \emph{after we have restricted in space} to outside the blow-up core.

Consider a solution $u(t)$ to \eqref{E:NLS} in the Merle-Rapha\"el class \eqref{E:MR-class}, let $T_0>0$ be the threshold time, $T>T_0$ the blow-up time and $x_0$ the blow-up point, as described in the introduction.  Our analysis focuses on the time interval $[T_0,T)$ on which the log-log asymptotics \eqref{E:log-log-rate} kick in.
Apply a space-time (rescaling) shift, in which $x=x_0$ is sent to $x=0$ and the time interval $[T_0,T)$ is sent to $[0,1)$, to obtain a transformed solution which we henceforth still denote by $u(t)$.  Now the blow-up time is $T=1$, the blow-up point is $x=0$, and \eqref{E:log-log-rate} becomes\footnote{
The rescaling is the following.  If we take $u(x,t)$ in the original frame (for $T_0\leq t<T$), and let $u(x,t) = \mu^{d/2}v(\mu(x-x_0),\mu^2(t-T_0))$ with $\mu=(T-T_0)^{-1/2}$, then $v(y,s)$ is defined in the modified frame (for $0\leq s<1$).  Moreover, we have $\|\nabla v(s) \|_{L_x^2} \sim (\log|\log\mu^{-2}(1-s)|)^{1/2}(1-s)^{-1/2}$, so now the implicit constant of comparability in \eqref{E:log-log-rescaled} depends on $T-T_0$.}
\begin{equation}
\label{E:log-log-rescaled}
\|\nabla u(t)\|_{L_x^2} \sim \left( \frac{\log|\log(1-t)|}{1-t} \right)^{1/2},
\end{equation}
which is now valid for all $0\leq t< 1$.  Note that now, however, the time $t=0$ ``initial--data,'' which we henceforth denote $u_0$, does not correspond to the original initial--data $u_0$ in Theorem \ref{T:main}.  We remark that the estimate \eqref{E:remainder-bounds} on the remainder $\tilde u(t)$ becomes
\begin{equation}
\label{E:remainder-bounds2}
\| \nabla \tilde u(t) \|_{L_x^2} \lesssim \frac{1}{(1-t)^{1/2}\log(1-t)} \,.
\end{equation}
In our analysis, the norm $L_I^\infty L_x^2$ for an interval $I=[0,T']$, $T'<T$, will be replaced by the norm $X_{0,\frac12+}(I)$.  While we have, from Lemma \ref{L:X-Strichartz}, the bound
$$
\|u\|_{L_I^\infty L_x^2} \lesssim \|u\|_{X_{0,\frac12+}(I)},
$$
the reverse bound does not in general hold.  Nevertheless, \eqref{E:log-log-rescaled} indicates that the solution is blowing-up close to the scale rate $(1-t)^{-1/2}$.  Thus, the local theory combined with \eqref{E:log-log-rescaled} implies a bound on $\|u\|_{X_{1,\frac12+}(I)}$, where $\log|\log(1-T')|$ is weakened to $(1-T')^{-\delta}$.

\begin{lemma}
\label{L:local}
For $I=[0,T']$ with $T'<T$, for $0<s\leq 1$, we have
$$
\|u\|_{X_{s,\frac12+}(I)} \leq c_s (1-T')^{-s(1+\delta)/2}   %+\delta
$$
with $c_s\nearrow +\infty$ as $s\searrow 0$.
\end{lemma}
The fact that $c_s$ diverges as $s\searrow 0$ results from the fact that \eqref{E:NLS} is $L^2$-critical, and thus, the local theory estimates break down at $s=0$.  At the technical level, some slack is needed in applying the Strichartz and bilinear Strichartz estimates, hence, need to take $b=\frac12-\delta$ in place of $b=\frac12+\delta'$.
\begin{proof}
We just carry out the argument for $s=1$.
Let $\lambda(t) = \|\nabla u(t)\|_{L^2}^{-1}$.  Let $s_k$ be the increasing sequence of times\footnote{One of the conclusions of the Merle-Rapha\"el analysis is the almost monotonicity
$$
\forall \; t_2 \geq t_1 \,, \qquad \lambda(t_2)< 2\lambda(t_1)
$$
of the scale parameter $\lambda(t) = \|\nabla u(t)\|_{L^2}^{-1}$.}
 such that $\lambda(s_k) = 2^{-k}$, so that $\|\nabla u(t)\|_{L^2}$ doubles over $[s_k,s_{k+1}]$.  From \eqref{E:log-log-rescaled}, we compute that $s_k = 1- 2^{-2k}\log k$. Note that $s_{k+1}-s_k \approx 2^{-2k} \, \log k$. Hence, we can rescale the cutoff solution $u(t)$ on the time interval $[s_k,s_{k+1}]$ to a solution $u'$ on the time interval $[0,\log k]$ so that $\|u'\|_{L_{[0,\log k]}^\infty H_x^1} \sim 1$.  We invoke the local theory over $\sim \log k$ time intervals $J$ each of unit size to obtain $\|u'\|_{X_{1,\frac12+}(J)} \sim 1$, which are square summed to obtain
$\|u'\|_{X_{1,\frac12+}(0,\log k)} \sim (\log k)^{1/2}$.  Returning to the original frame of reference, we conclude that
$$
\| u\|_{X_{1,\frac12+}(s_k,s_{k+1})} \lesssim 2^{k(1+\delta)} \,,
$$
where a $\delta$-loss is incured in part from the $(\log k)^{1/2}$ factor but also from the $b=\frac12+\delta$ weight in the $X$-norm.
Thus,
$$
\|u\|_{X_{1,\frac12+}(0,s_K)} = \left( \sum_{k=1}^{K-1} 2^{2k(1+\delta)} \right)^{1/2} \sim 2^{K(1+\delta)}. %\delta
$$
\end{proof}

Now suppose that $u(t)$ satisfies \eqref{E:log-log-rescaled}.   Let $t_k=1-2^{-k}$ and $I_k = [0,t_k]$.  Then from \eqref{E:log-log-rescaled} and mass conservation, we have
\begin{equation}
\label{E:crude-0}
\|P_{\geq N} u(t) \|_{L_{I_k}^\infty L_x^2} \lesssim
\begin{cases}
2^{k(1+\delta)/2}%(\log k)^{1/2}
N^{-1} & \text{for }N\geq  2^{k(1+\delta)/2}%(\log k)^{1/2}
\\
1 & \text{for } N \leq  2^{k(1+\delta)/2}. %(\log k)^{1/2}.
\end{cases}
\end{equation}
To refine \eqref{E:crude-0}, we will work with local-theory estimates, and thus, use the analogous bound on the Bourgain norm $X_{0,\frac12+}(I_k)$.  From Lemma \ref{L:local} we obtain
\begin{equation}
\label{E:crude}
\| P_{\geq N} u\|_{X_{0,\frac12+}(I_k)} \lesssim N^{-s} \|P_{\geq N}  u\|_{X_{s,\frac12+}(I_k)} \leq c_s N^{-s} 2^{ks(1+\delta)/2}\,.
\end{equation}
We obtain from \eqref{E:crude} that
\begin{equation}
\label{E:crude-2}
\|P_{\geq N} u \|_{X_{0,\frac12+}(I_k)} \lesssim
\begin{cases}
2^{k(1+\delta)/2}N^{-1} & \text{for } N \geq 2^{k(1+\delta)/2}  \\
2^{k\delta'} & \text{for }N \leq 2^{k(1+\delta)/2}.
\end{cases}
\end{equation}

The next step is to run local-theory estimates to improve \eqref{E:crude-2} at \emph{high} frequencies.
Frequencies   $N\lesssim 2^k \sim (1-t_k)^{-1}$ on $I_k$ effectively do not make it out of the blow-up core before blow-up time due to the finite speed of propagation for such frequencies.\footnote{Recall that for the Schr\"odinger equation, frequencies of size $N$ propagate at speed $N$, and thus, travel a distance $O(1)$ in time $N^{-1}$.}  Hence, these \emph{low} frequencies can be controlled by spatial location, which we address in \S \ref{S:finite-speed}.    On the other hand, \eqref{E:crude-2} shows that the solution at frequencies $N\gtrsim 2^{k(1+\delta)/2}$ is small.  Thus, for these \emph{high} frequencies, dispersive estimates might be able, upon iteration, to show that the solution is even smaller at these high frequencies.

To chose an intermediate dividing point between the high frequencies that are capable of exiting the blow-up core before blow-up time ($N\gtrsim 2^k$) and the frequency scale at which the blow-up is taking place ($N\sim 2^{k/2}(\log k)^{1/2}$), we consider frequencies $\geq 2^{3k/4}$ to be \emph{high} frequencies and frequencies $\leq 2^{3k/4}$ to be \emph{low} frequencies.  The goal of this section is Prop. \ref{P:high-freq} below, which shows that the high frequencies are bounded in $H^1$.  In \S \ref{S:finite-speed} below, we will localize in space to the external region and then control the low frequencies.

We first address the dimension $d=1$ case.
%Note that \eqref{E:crude-2} verifies the hypothesis of Lemma
%\ref{L:iterator} with $\alpha(k,N) = 2^{k(1+\delta)/2}N^{-1}$.

\begin{lemma}[high frequency recurrence, 1d]
\label{L:iterator}
Take $d=1$. Let $t_k=1-2^{-k}$ and $I_k = [0,t_k]$.   Let $u(t)$ be a solution such that \eqref{E:log-log-rescaled} holds, and define
\begin{equation}
\label{E:initial-bounds}
\alpha(k,N) \defeq \|P_{\geq N} u \|_{X_{0,\frac12+}(I_k)} \,.
\end{equation}
Then there exists an absolute constant $0<\mu\ll 1$ such that for $N\geq 2^{k(1+\delta)/2}$,
\begin{equation}
\label{E:step-160}
\|P_{\geq N} (u - e^{it\partial_x^2}u_0) \|_{X_{0,\frac12+}(I_k)} \lesssim 2^{k(1+\delta)/2}N^{-1+\delta}\alpha(k+1,\mu N) + 2^{k\delta} \alpha(k+1, \mu N)^2.
\end{equation}
In particular, by \ref{L:flow},
\begin{equation}
\label{E:step-161}
\alpha(k,N) \lesssim \|P_{\geq N}u_0 \|_{L_x^2} + 2^{k(1+\delta)/2}N^{-1+\delta}\alpha(k+1,\mu N) + 2^{k\delta} \alpha(k+1, \mu N)^2.
\end{equation}
\end{lemma}
\begin{proof}
By Lemma \ref{L:flow} \eqref{E:step-101} with $\omega=2^{-k-1}$ and $I=I_k$,
$$
\|P_{\geq N} (u-e^{it\partial_x^2}u_0) \|_{X_{0,\frac12+}(I_k)} \lesssim  2^{k\delta} \|P_{\geq N} (|u|^4 u) \|_{X_{0,-\frac12+}(I_{k+1})} \,.
$$
In the rest of the proof, we estimate the right-hand side of the above estimate, and we will just write $I_k$ instead of $I_{k+1}$ for convenience.
By duality,
$$
\|P_{\geq N} (|u|^4 u) \|_{X_{0,-\frac12+}(I_k)} = \sup_{\|w\|_{X_{0,\frac12-}(I_k)}=1} \int_{I_k} \int_{x\in \mathbb{R}} P_{\geq N}(|u|^4u) \, w \, dx\, dt \,.
$$
Fix $w$ with $\|w\|_{X_{0,\frac12-}(I_k)}=1$ and let
$$
J \defeq \int_{I_k} \int_{x\in \mathbb{R}} P_{\geq N}(|u|^4u) \, w \, dx\, dt \,.
$$
Then $J$ can be decomposed into a finite sum of terms $J_\alpha$, each of the form (we have dropped complex conjugates, since they are unimportant in the analysis)
$$
J_\alpha \defeq \int_0^{t_k} \int_{x\in \mathbb{R}} P_{\geq N}(u_1 u_2 u_3 u_4 u_5) \, w \, dx\, dt
$$
such that each term (after a relabeling of the $u_j$, $1\leq j\leq 5$) falls into exactly one of the following two categories.
\footnote{Indeed, decompose each $u_j$ as $u_j=u_{j,\lo}+u_{j,\med}+u_{j,\hi}$, where $u_{j,\lo} = P_{\leq N/160}u_j$, $u_{j,\med}=P_{N/160 \leq \cdot \leq N/20}$, and $u_{j,\hi} = P_{\geq N/20} u_j$.  Then in the expansion of $u_1u_2u_3u_4u_5$, at least one term must be ``hi''; without loss take this to be $u_5$.  Case 1 corresponds to $u_{1,\lo}u_{2,\lo}u_{3,\lo}u_{4,\lo}u_{5,\hi}$ and Case 2 corresponds to everything else (at least one $u_j$, for $1\leq j \leq 4$, must be ``med'' or ``hi''.  Hence, we can take $\mu=\frac{1}{160}$.}

Note that $w$ is frequency supported in $|\xi|\gtrsim N$.

\noindent\emph{\underline{Case 1} (exactly one high)}.  Each $u_j$ for $1\leq j\leq 4$ is frequency supported in $|\xi|\leq \mu N$ and $u_5$ is frequency supported in $|\xi|\geq 8 \mu N$.  In this case, we estimate as
\begin{equation}
\label{E:step-0}
|J_\alpha| \leq \|u_1\|_{L_{I_k}^\infty L_x^\infty} \|u_2\|_{L_{I_k}^\infty L_x^\infty} \|u_3 u_5\|_{L_{I_k}^2 L_x^2} \|u_4 w \|_{L_{I_k}^2 L_x^2}\,.
\end{equation}
For $j=1,2$, Gagliardo-Nirenberg and \eqref{E:log-log-rescaled} implies
\begin{equation}
\label{E:step-1}
\|u_j\|_{L_{I_k}^\infty L_x^\infty} \lesssim \|u_j\|_{L_{I_k}^\infty L_x^2}^{1/2} \|\partial_x u_j\|_{L_{I_k}^\infty L_x^2}^{1/2} \lesssim 2^{k(1+\delta)/4} \,.
\end{equation}
The bilinear Strichartz estimate (Lemma \ref{L:X-bilinear-Strichartz})  yields
\begin{equation}
\label{E:step-2}
\|u_3 u_5\|_{L_{I_k}^2 L_x^2} \lesssim N^{-1/2} \|u_3 \|_{X_{0,\frac12+}(I_k)} \|u_5\|_{X_{0,\frac12+}(I_k)}\lesssim N^{-1/2} 2^{k\delta} \alpha(k,\mu N).
\end{equation}
The interpolated bilinear Strichartz estimate (Lemma \ref{L:interp-2}) yields
\begin{equation}
\label{E:step-3}
\|u_4 w \|_{L_{I_k}^2 L_x^2} \lesssim N^{-\frac12+\delta} \|u_4\|_{X_{0,\frac12+}(I_k)} \|w\|_{X_{0,\frac12-}(I_k)} \lesssim N^{-\frac12+\delta} 2^{k\delta}.
\end{equation}
Substituting \eqref{E:step-1}, \eqref{E:step-2}, \eqref{E:step-3} into \eqref{E:step-0}, we obtain
$$
|J_\alpha| \lesssim 2^{k(1+\delta)/2}N^{-1+\delta}\alpha(k,\mu N).
$$

\noindent\emph{\underline{Case 2} (at least two high)}.  Both $u_4$ and $u_5$ are frequency supported in $|\xi|\geq \mu N$ (no restrictions on $u_j$ for $1\leq j\leq 3$).  Then we estimate as
\begin{equation}
\label{E:step-9}
|J_\alpha| \leq \|u_1\|_{L_{I_k}^6L_x^{6+\delta}}\|u_2\|_{L_{I_k}^6L_x^6} \|u_3\|_{L_{I_k}^6L_x^6} \|u_4\|_{L_{I_k}^6L_x^6} \|u_5\|_{L_{I_k}^6L_x^6} \|w\|_{L_{I_k}^6L_x^{6-\delta'}}\,.
\end{equation}
For $2\leq j \leq 3$ we invoke the Strichartz estimate (Lemma \ref{L:X-Strichartz}) and \eqref{E:crude-2} to obtain
\begin{equation}
\label{E:step-10}
\|u_j\|_{L_{I_k}^6L_x^6} \lesssim \|u_j \|_{X_{0,\frac12+}(I_k)} \leq 2^{k\delta}\,.
\end{equation}
For $4\leq j \leq 5$ we invoke the Strichartz estimate (Lemma \ref{L:X-Strichartz}) and \eqref{E:initial-bounds} to obtain
\begin{equation}
\label{E:step-10b}
\|u_j\|_{L_{I_k}^6L_x^6} \lesssim \|u_j \|_{X_{0,\frac12+}} \leq \alpha(k,\mu N)\,.
\end{equation}
For $j=1$, by Sobolev embedding, the Strichartz estimate (Lemma \ref{L:X-Strichartz}), and \eqref{E:crude-2},
\begin{equation}
\label{E:step-11}
\|u_1 \|_{L_{I_k}^6L_x^{6+}} \lesssim \|D_x^\delta u_1 \|_{L_{I_k}^6L_x^6} \lesssim \|u_1 \|_{X_{\delta, \frac12+}(I_k)} \lesssim 2^{k\delta}\,.
\end{equation}
By the interpolated Strichartz estimate (Lemma \ref{L:interp-1}), we have
\begin{equation}
\label{E:step-12}
\|w\|_{L_t^6 L_x^{6-}} \lesssim \|w\|_{X_{0,\frac12-}(I_k)}=1\,.
\end{equation}
Using \eqref{E:step-10}, \eqref{E:step-10b}, \eqref{E:step-11}, \eqref{E:step-12}, in \eqref{E:step-9},
$$
|J_\alpha|\lesssim  2^{k\delta}\alpha(k,\mu N)^2.
$$
\end{proof}

In the 2d case, we will just go ahead and assume that $N\geq 2^{3k/4}$ to reduce confusion with $\delta$'s.

\begin{lemma}[high frequency recurrence, 2d]
\label{L:iterator-2d}
Take $d=2$. Let $t_k=1-2^{-k}$ and $I_k = [0,t_k]$.   Let $u(t)$ be a solution such that \eqref{E:log-log-rescaled} holds and define
\begin{equation}
\label{E:initial-bounds-2d}
\alpha(k,N) \defeq \|P_{\geq N} u \|_{X_{0,\frac12+}(I_k)} \,.
\end{equation}
Then there exists an absolute constant $0<\mu\ll 1$ such that for $N\gtrsim 2^{3k/4}$,
\begin{equation}
\label{E:step-162}
\|P_{\geq N} (u - e^{it\Delta}u_0) \|_{X_{0,\frac12+}(I_k)} \lesssim 2^{k\delta}N^{-\frac16+\delta} \alpha(k+1,\mu N).
\end{equation}
In particular, by Lemma \ref{L:flow},
\begin{equation}
\label{E:step-163}
\alpha(k,N) \lesssim \|P_{\geq N} u \|_{L_x^2} + 2^{k\delta}N^{-\frac16+\delta} \alpha(k+1,\mu N) \,.
\end{equation}
\end{lemma}
\begin{proof}
By Lemma \ref{L:flow} \eqref{E:step-101} with $I=I_k$ and $\omega=2^{-k-1}$,
$$
\|P_{\geq N} (u-e^{it\Delta}u_0) \|_{X_{0,\frac12+}(I_k)} \lesssim  2^{k\delta} \|P_{\geq N} (|u|^2 u) \|_{X_{0,-\frac12+}(I_{k+1})} \,.
$$
In the remainder of the proof, we estimate the right-hand side, and for convenience take $I_{k+1}$ to be $I_k$.
By duality,
$$
\|P_{\geq N} (|u|^2 u) \|_{X_{0,-\frac12+}(I_k)} = \sup_{\|w\|_{X_{0,\frac12-}(I_k)}=1} \int_{I_k} \int_{x\in \mathbb{R}} P_{\geq N}(|u|^2u) \, w \, dx\, dt \,.
$$
Fix $w$ with $\|w\|_{X_{0,\frac12-}(I_k)}=1$ and let
$$
J \defeq \int_{I_k} \int_{x\in \mathbb{R}} P_{\geq N}(|u|^2u) \, w \, dx\, dt \,.
$$
Then $J$ can be decomposed into a finite sum of terms $J_\alpha$, each of the form (we have dropped complex conjugates, since they are unimportant in the analysis)
$$
J_\alpha \defeq \int_0^{t_k} \int_{x\in \mathbb{R}} P_{\geq N}(u_1 u_2 u_3) \, w \, dx\, dt
$$
such that each term (after a relabeling of the $u_j$, $1\leq j\leq 3$) falls into exactly one of the following two categories.\footnote{Indeed, decompose $u_j=u_{j,\lo}+u_{j,\med}+u_{j,\hi}$, where $u_{j,\lo} = P_{\leq N^{5/6}}u_j$, $u_{j,\med} = P_{N^{5/6}\leq \cdot \leq \frac1{12}N}$, and $u_{j,\hi} = P_{\geq \frac1{12}N} u_j$.  Then at least one term must be ``hi''; take it to be $u_3$.  Case 1 corresponds to $u_{1,\lo}u_{2,\lo}u_{3,\hi}$ and Case 2 corresponds to all other possibilities.  Hence, we can take $\mu=\frac{1}{12}$.}  Note that $w$ is frequency supported in $|\xi|\gtrsim N$.

\noindent\textit{\underline{Case 1} (exactly one high)}.  Both $u_1$ and $u_2$ are frequency supported in $|\xi|\leq N^{5/6}$ and $u_3$ is frequency supported in $|\xi|\geq \frac1{12}  N$.  In this case, we estimate as
$$
|J_\alpha| \lesssim \|u_1 w\|_{L_{I_k}^2L_x^2} \|u_2 u_3 \|_{L_{I_k}^2L_x^2} \,.
$$
By the interpolated bilinear Strichartz estimate (Lemma \ref{L:interp-2}),
$$
\|u_1 w\|_{L_{I_k}^2L_x^2} \lesssim (N^{5/6})^{1/2}N^{-\frac12+\delta} \|u_1\|_{X_{0,\frac12-}(I_k)} \|w\|_{X_{0,\frac12-}(I_k)} \lesssim N^{-\frac1{12}+\delta} 2^{k\delta} \,,
$$
and by Lemma \ref{L:X-bilinear-Strichartz} directly,
$$
\|u_2 u_3 \|_{L_{I_k}^2L_x^2}\lesssim (N^{5/6})^{1/2} N^{-\frac12+\delta} \|u_2 \|_{X_{0,\frac12+}(I_k)} \|u_3\|_{X_{0,\frac12+}(I_k)} \lesssim N^{-\frac1{12}+\delta} 2^{k\delta} \alpha(k,\mu N)\,.
$$
Combining yields
$$
|J_\alpha| \lesssim N^{-\frac16+\delta} 2^{k\delta} \alpha(k,\mu N) \,.
$$

\noindent\textit{\underline{Case 2} (at least two high)}.  Here we suppose that $u_2$ is frequency supported in $|\xi|\geq N^{5/6}$ and $u_3$ is frequency supported in $|\xi|\geq \mu N$; we make no assumptions about $u_1$.
Then we estimate as
$$
|J_\alpha| \lesssim \|u_1 \|_{L_{I_k}^4L_x^{4+\delta}} \|u_2\|_{L_{I_k}^4L_x^4} \|u_3\|_{L_{I_k}^4L_x^4} \|w\|_{L_{I_k}^4L_x^{4-\delta}} \,.
$$
For $u_1$, we use Sobolev embedding and \eqref{E:crude-2} to obtain
$$
\|u_1\|_{L_{I_k}^4L_x^{4+\delta}} \lesssim \|D_x^\delta u_1 \|_{L_{I_k}^4L_x^4} \lesssim \| u_1 \|_{X_{\delta, \frac12+}(I_k)} \lesssim 2^{k\delta} \,.
$$
Since $N\gtrsim 2^{3k/4}$, we have $N^{5/6} \gtrsim 2^{5k/8} \gg 2^{k(1+\delta)/2}$, and thus by Lemma \ref{L:X-Strichartz} and \eqref{E:crude-2},
\begin{align*}
\|u_2 \|_{L_{I_k}^4L_x^4}
&\lesssim 2^{k(1+\delta)/2}N^{-5/6} \\
&\lesssim (2^{k(1+\delta)} N^{-2/3}) N^{-1/6} \\
&\lesssim 2^{k\alpha}N^{-1/6}, & \text{since }N \gtrsim 2^{3k/4}\,.
\end{align*}
For $u_3$, we use  Lemma \ref{L:X-Strichartz} and \eqref{E:initial-bounds-2d} to obtain
$$
\|u_3 \|_{L_{I_k}^4L_x^4} \lesssim \alpha(k,\mu N).
$$
Combining, we obtain (changing $\delta$'s)
$$
|J_\alpha| \lesssim 2^{k\delta}N^{-1/6} \alpha(k, \mu N) \,.
$$
\end{proof}

The main result of this section is the following.  It states that high frequencies (those strictly above  $2^{3k/4}$) are $H^1$ bounded on $I_k$.  Moreover, if we subtract the linear flow, we obtain $H^{\frac43-\delta}$ boundedness for frequencies above $2^{3k/4}$ in the case $d=1$ and $H^{\frac76-\delta}$ boundedness for frequencies above $2^{3k/4}$ in the case $d=2$.\footnote{
In fact, the threshold $\geq 2^{3k/4}$, to obtain $H^1$ boundedness (but not \eqref{E:step-40}), can be replaced by $2^{k(1+\delta)/2}$ for any $\delta>0$; in the $d=1$ case, one can appeal to Lemma \ref{L:iterator} with a strictly smaller choice of $\delta$ in order to obtain a nontrivial gain upon each application of Lemma \ref{L:iterator}.  The number of applications of Lemma \ref{L:iterator} is still finite number but $\delta$-dependent.  In the 2d case, Lemma \ref{L:iterator-2d} would first need to be rewritten.  We have stated the proposition with threshold $\geq 2^{3k/4}$ because this is all that is needed in \S \ref{S:finite-speed}, and it allows us to avoid confusion with multiple small parameters.}

\begin{proposition}
\label{P:high-freq}
Let $t_k=1-2^{-k}$, $I_k = [0,t_k]$, and let $u(t)$ be a solution to \eqref{E:NLS} such that \eqref{E:log-log-rescaled} holds.  Then we have
$$
\| P_{\geq 2^{3k/4}} u(t)\|_{L^\infty_{I_k} H_x^1} \lesssim \| P_{\geq 2^{3k/4}} u(t)\|_{X_{1,\frac12+}(I_k)} \lesssim 1 \,.
$$
Moreover, we have the following regularity \emph{above} $H^1$ after the linear flow of the initial data is removed:  For any $0\leq s\leq \frac{4}{3}-\delta$ in the case $d=1$ and for any $0\leq s \leq \frac76-\delta$ in the case $d=2$, we have
\begin{equation}
\label{E:step-40}
\| P_{\geq 2^{3k/4}} (u(t)-e^{it\Delta}u_0) \|_{L_{I_k}^\infty H_x^s} \lesssim  \| P_{\geq 2^{3k/4}} (u(t)-e^{it\Delta}u_0) \|_{X_{s,\frac12+\delta}(I_k)} \lesssim 1.
\end{equation}
\end{proposition}
\begin{proof}
We carry out the $d=1$ case in full, which is a consequence of Lemma \ref{L:iterator}.  The $d=2$ case follows from Lemma \ref{L:iterator-2d} in a similar way.

By \eqref{E:crude-2}, we start with the knowledge that $\alpha(k,N) \lesssim 2^{k(1+\delta)/2}N^{-1}$ for $N \geq 2^{k(1+\delta)/2}$.  Note
$$
\|P_{\geq N} u_0 \|_{L_x^2}\lesssim N^{-1} \|\nabla u_0 \|_{L_x^2} \lesssim N^{-1}\,.
$$
By \eqref{E:step-161} in Lemma \ref{L:iterator},
\begin{equation}
\label{E:step-140}
\alpha(k,N) \lesssim N^{-1} +  2^{k(1+\delta)/2}N^{-1+\delta} \alpha(k+1,\mu N) \,.
\end{equation}
Application of \eqref{E:step-140} $J$ times gives
$$
\alpha(k,N) \lesssim N^{-1} \left(\sum_{j=0}^{J-1} (2^{k(1+\delta)/2}N^{-1+\delta})^j\right) + (2^{k(1+\delta)/2}N^{-1+\delta})^J \alpha(k+J,\mu^JN)\,.
$$
Since $N\geq 2^{3k/4}$, we have $2^{k/2}N^{-1} \lesssim N^{-1/3}$.  Taking $J=7$ we obtain
$$
\alpha(k,N) \lesssim N^{-1} \,.
$$
Substituting this \eqref{E:step-160} of Lemma \ref{L:iterator}, we obtain
$$
\| P_{\geq N}(u(t)-e^{it\partial_x^2}u_0) \|_{X_{0,\frac12+}(I_k)} \lesssim 2^{k(1+\delta)/2}N^{-2+\delta} \lesssim N^{-\frac43+\delta},
$$
yielding the claim.
\end{proof}

\section{Finite speed of propagation}
\label{S:finite-speed}

Recall that the main result of the last section was Prop. \ref{P:high-freq}, which showed that the solution at frequencies $\geq 2^{3k/4}$ is $H^1$ bounded on $I_k$.  This was achieved without applying any restriction in space.  In this section, we apply a spatial restriction to $|x|\geq R$ (outside the blow-up core), and study the low frequencies $\leq 2^{3k/4}$ on $I_k$.  Since frequencies of size $N$ propagate at speed $N$, and thus, travel a distance $O(1)$ over a time $N^{-1}$, we expect that frequencies of size $\lesssim 2^k$ involved in the blow-up dynamics will be incapable of exiting the blow-up core $|x|\leq R$ before blow-up time.

Since $I_k=[0,t_k]$ and $t_k=1-2^{-k}$, restricting to frequencies $\leq 2^{3k/4}$ on $I_k$, for each $k$, is effectively equivalent to inserting a time-dependent spatial frequency projection $P_{\leq (1-t)^{-3/4}}$.  The main technical Lemma \ref{L:low-freq-it} below  shows that, for $0<r_1<r_2<\infty$, the $H^s$ size of the solution in the external region $|x|\geq r_2$ is bounded by the $H^{s-\frac18}$ size of the solution in the slightly larger external region $|x|\geq r_1$.  This lemma is proved by studying the equation solved by $P_{\leq (1-t)^{-3/4}}\psi u$, where $\psi$ is a spatial cutoff.  In estimating the inhomogeneous terms of this equation, we use that the presence of the $P_{\leq (1-t)^{-3/4}}$ projection enables an exchange of $\alpha$ spatial derivatives for a factor of $(1-t)^{-3\alpha/4}$.  This is the manner in which finite-speed of propagation is implemented.  Lemma \ref{L:low-freq-it} is the main recurrence device for proving Prop. \ref{P:finite-speed-new}, giving the $H^1$ boundedness of the solution in the external region, completing the proof of Theorem \ref{T:main}.

Before getting to Lemma \ref{L:low-freq-it}, we begin by using the method of Rapha\"el \cite{Rap}, based on the use of local smoothing and \eqref{E:remainder-bounds2}, to achieve a small gain of regularity.\footnote{In the $d=1$ case, we obtain a gain of $\frac25$ derivatives in this first step, but in fact the proof could be rewritten to achieve a gain of $s<\frac12$ derivatives.  The reason $s=\frac12$ derivatives cannot be achieved in one step is the failure of the $H^{1/2}\hookrightarrow L^\infty$ embedding needed to estimate the nonlinear term.  One could achieve $\frac12$ derivatives by running the same argument twice, but this is unnecessary since we only need a small gain of $s>0$ to complete the proof of our main new Lemma \ref{L:low-freq-it}/Prop. \ref{P:finite-speed-new} below, which enables us to reach the full $s=1$ gain.   One cannot achieve a gain of $s>\frac12$ by the method employed in the proof of Lemma \ref{L:a-little-1} alone due to the term $\partial_x (\psi_R' \, u)$.}

\begin{lemma}[a little regularity, $d=1$ case]
\label{L:a-little-1}
Suppose $d=1$.  Suppose that $u(t)$ solving \eqref{E:NLS} with $H^1$ initial data satisfies \eqref{E:log-log-rescaled}.  Fix $R>0$.  Then
$$
\|\la D_x \ra^{2/5} \psi_R u\|_{L_{[0,1)}^\infty L_x^2} \lesssim 1 \,,
$$
where $\psi_R(x)=\psi(x/R)$ and $\psi(x)$ is a smooth cutoff with $\psi(x)=1$ for $|x|\geq \frac12$ and $\psi(x)=0$ for $|x|\leq \frac14$.
\end{lemma}
\begin{proof}
Let $w=\psi_R u$ and $q=\psi_{R/2}u$.  Then $w$ solves the equation
\begin{align*}
i\partial_t w + \partial_x^2 w
&= -|q|^4 w + 2\partial_x ( \psi_R'  \, u) - \psi_R'' \, u \\
&= F_1 +F_2 +  F_3 \,.
\end{align*}
Apply $\la D_x \ra^{2/5}$, and estimate with $I=[T_1,1)$ using the (dual) local smoothing estimate for the $F_2$ term,
$$
\|\la D_x\ra^{2/5} w \|_{L_I^\infty L_x^2} \lesssim
\begin{aligned}[t]
&\|\la D_x\ra^{2/5} w(T_1)\|_{L_x^2} + \| \la D_x\ra^{2/5}F_1 \|_{L_I^1 L_x^2} \\
&+ \| \la D_x\ra^{2/5}\la D_x\ra ^{-1/2} F_2 \|_{L_I^2 L_x^2} + \| \la D_x\ra^{2/5} F_3 \|_{L_I^1L_x^2}\,.
\end{aligned}
$$
We begin by estimating term $F_1$.  By the fractional Leibniz rule,
\begin{align*}
\|D_x^{2/5} F_1 \|_{L_I^1L_x^2}
&\lesssim \| |q|^4 \|_{L_I^1 L_x^\infty} \|D_x^{2/5} w \|_{L_I^\infty L_x^2} + \|D_x ^{2/5}|q|^4 \|_{L_I^1 L_x^{5/2}} \|w \|_{L_I^\infty L_x^{10}} \,.\\
&\lesssim (\| |q|^4 \|_{L_I^1 L_x^\infty} + \|D_x^{2/5}|q|^4 \|_{L_I^1 L_x^{5/2}}) \|D_x^{2/5} w \|_{L_I^\infty L_x^2} \,.
\end{align*}
By Sobolev/Gagliardo-Nirenberg embedding and \eqref{E:remainder-bounds2},
$$
\| |q|^4 \|_{L_x^\infty} + \| D_x^{2/5} |q|^4 \|_{L_x^{5/2}} \lesssim \|q\|_{L_x^2}^2 \|\partial_x q \|_{L_x^2}^2 \lesssim (1-t)^{-1}(\log (1-t)^{-1})^{-2}  \,.
$$
Applying the $L_I^1$ time norm, we obtain a bound by $(\log (1-T_1)^{-1})^{-1}$.  Hence,
$$
\| \la D_x\ra^{2/5} F_1 \|_{L_I^1L_x^2} \lesssim (\log (1-T_1)^{-1})^{-1} \| \la D_x\ra^{2/5} w \|_{L_I^\infty L_x^2} \,.
$$
Next, we address term $F_2$.  We have
\begin{align*}
\| \la D_x\ra^{2/5}\la D_x\ra^{-1/2} F_2 \|_{L_I^2L_x^2}
&\lesssim \| \la D_x \ra^{9/10} q \|_{L_I^2L_x^2} \\
&\lesssim \| q \|_{L_I^\infty L_x^2}^{1/10} \| \|\la \partial_x \ra q \|_{L_x^2}^{9/10} \|_{L_I^2} \,.
\end{align*}
From \eqref{E:remainder-bounds2}, we have $\| \partial_x q \|_{L_x^2} \lesssim (T-t)^{-1/2}|\log(1-t)|^{-1}$, and hence,
$$
\| \la D_x\ra^{2/5}\la D_x\ra^{-1/2} F_2 \|_{L_I^2L_x^2} \lesssim (1-T_1)^{1/10} \,.
$$
Term $F_3$ is comparatively straightforward.   Indeed, we obtain
\begin{align*}
\| \la D_x \ra^{2/5} F_3 \|_{L_I^1L_x^2}
&\lesssim \| u \|_{L_I^\infty L_x^2}^{3/5}\|  \| \la \partial_x \ra \psi_2 u \|_{L_x^2}^{2/5}\|_{L_I^1} \\
&\lesssim (1-T_1)^{4/5} \,.
\end{align*}
Collecting the above estimates, we obtain
$$
\|\la D_x \ra^{2/5} w \|_{L_I^\infty L_x^2} \lesssim \| \la D_x \ra^{2/5} w(T_1) \|_{L_x^2} + (\log(1-T_1)^{-1})^{-1} \|\la D_x \ra^{2/5} w \|_{L_I^\infty L_x^2} + (1-T_1)^{1/10} \,.
$$
By taking $T_1$ sufficiently close to $1$ so that $(\log(1-T_1)^{-1})^{-1}$ beats out the (absolute) implicit constants furnished by the estimates, we obtain
$$
\|\la D_x \ra^{2/5} w \|_{L_I^\infty L_x^2} \lesssim \| \la D_x \ra^{2/5} w(T_1) \|_{L_x^2} + (1-T_1)^{1/10} \,,
$$
which yields the claim.
\end{proof}

\begin{lemma}[a little regularity, $d=2$ case]
\label{L:a-little-2}
Suppose $d=2$.  Suppose that $u(t)$ solving \eqref{E:NLS} with $H^1$ initial data satisfies \eqref{E:log-log-rescaled}.  Fix $R>0$.  Then
$$
\|\la D_x \ra^{1/2} \psi_R u\|_{L_{[0,1)}^\infty L_x^2} \lesssim 1 \,,
$$
where $\psi_R(x)=\psi(x/R)$ and $\psi(x)$ is a smooth cutoff with $\psi(x)=1$ for $|x|\geq \frac12$ and $\psi(x)=0$ for $|x|\leq \frac14$.
\end{lemma}
\begin{proof}
Let $w=\psi_R u$ and $q=\psi_{R/2}u$, and take $\tilde \psi = \nabla_x \psi_R$ and $\tilde{\tilde \psi} = \Delta_x \psi_R$.  Then $w$ solves the equation
\begin{align*}
i\partial_t w + \Delta w
&= -|q|^2 w + 2\nabla_x \cdot ( \tilde \psi \, u) - \tilde{\tilde \psi} \, u \\
&= F_1 +F_2 +F_3\,.
\end{align*}
Apply $\la D_x\ra ^{1/2}$, and estimate with $I=[T_1,1)$ using the (dual) local smoothing estimate for the term $F_2$,
\begin{align*}
\indentalign \|\la D_x\ra ^{1/2} w \|_{L_I^\infty L_x^2} + \|\la D_x\ra ^{1/2} w \|_{L_I^4 L_x^4}  \\
&\lesssim \|\la D_x\ra ^{1/2}w_0\|_{L_x^2} + \|\la D_x\ra^{1/2} F_1 \|_{L_I^{4/3} L_x^{4/3}} + \|F_2  \|_{L_I^2 L_x^2} + \| \la D_x \ra^{1/2} F_3 \|_{L_I^1L_x^2} \,.
\end{align*}

Before we begin treating term $F_1$, let us note that \eqref{E:remainder-bounds2}, $\|\nabla q \|_{L_x^2} \lesssim (1-t)^{-1/2}(\log(1-t)^{-1})^{-1}$, and hence, $\|\nabla q\|_{L_I^2 L_x^2} \lesssim (\log(1-T_1)^{-1})^{-1/2}$.
By the fractional Leibniz rule and Sobolev/Gagliardo-Nirenberg embedding,
$$
\|D_x^{1/2} |q|^2 \|_{L_x^2} \lesssim \|D_x^{1/2} q \|_{L_x^4} \|q\|_{L_x^4} \lesssim \| q\|_{L_x^2}^{1/2} \|\nabla q \|_{L_x^2}^{3/2} \,.
$$
Hence,
\begin{equation}
\label{E:step-130}
\| D_x^{1/2} |q|^2 \|_{L_I^{4/3} L_x^2} \lesssim \| q\|_{L_I^\infty L_x^2}^{1/2} \|\nabla q\|_{L_I^2 L_x^2}^{3/2} \lesssim (\log (1-T_1)^{-1})^{-3/4}\,.
\end{equation}
Also, we have
$$
\|q \|_{L_x^4} \lesssim \|D_x^{1/2} q \|_{L_x^2} \lesssim \|q \|_{L_x^2}^{1/2}\|\nabla q\|_{L_x^2}^{1/2} \,,
$$
and hence,
\begin{equation}
\label{E:step-131}
\| q \|_{L_I^4L_x^4}^2 \lesssim \| q\|_{L_I^\infty L_x^2} \| \nabla q \|_{L_I^2L_x^2} \lesssim (\log (1-T_1)^{-1})^{-1/2}\,.
\end{equation}

Now we proceed with the estimates for term $F_1$.  By the fractional Leibniz rule (in $x$),
$$
\| \la D_x\ra^{1/2} F_1 \|_{L_I^{4/3}L_x^{4/3}} \lesssim \| \la D_x \ra^{1/2} |q|^2 \|_{L_I^{4/3} L_x^2} \|w\|_{L_I^\infty L_x^4} + \| |q|^2 \|_{L_I^2L_x^2} \|\la D_x \ra^{1/2} w \|_{L_I^4L_x^4} \,.
$$
By \eqref{E:step-130} and \eqref{E:step-131}, we obtain
$$
\| \la D_x\ra^{1/2} F_1 \|_{L_I^{4/3}L_x^{4/3}} \lesssim (\log(1-T_1)^{-1})^{-1/2} (\| \la D_x\ra^{1/2} w \|_{L_I^\infty L_x^2}+ \| \la D_x\ra^{1/2} w \|_{L_I^4 L_x^4}) \,.
$$
Next, we treat the $F_2$ term.  Again since $\|\nabla q \|_{L_x^2} \lesssim (1-t)^{-1/2}(\log(1-t)^{-1})^{-1}$,
$$
\|F_2  \|_{L_I^2 L_x^2} \lesssim (\log(1-T_1)^{-1})^{-1} \,.
$$
The $F_3$ term is comparatively straightforward.

Collecting the above estimates, we have
\begin{align*}
\indentalign
\| \la D_x\ra^{1/2} w \|_{L_I^\infty L_x^2} + \| \la D_x \ra^{1/2} w \|_{L_I^4L_x^4} \\
&\lesssim
\begin{aligned}[t]
&\| \la D_x\ra^{1/2} w(T_1) \|_{L_x^2} + (\log(1-T_1)^{-1})^{-1} \\
&+(\log(1-T_1)^{-1})^{-1/2}(\| \la D_x\ra^{1/2} w \|_{L_I^\infty L_x^2}+ \|\la D_x\ra^{1/2} w \|_{L_I^4 L_x^4}) \,.
\end{aligned}
\end{align*}
By taking $T_1$ sufficiently close to $1$, we obtain
$$
\| \la D_x\ra^{1/2} w \|_{L_I^\infty L_x^2} \lesssim \| \la D_x\ra^{1/2} w(T_1) \|_{L_x^2} + (\log(1-T_1)^{-1})^{-1} \,,
$$
which yields the claim.
\end{proof}

\begin{lemma}[low frequency recurrence]
\label{L:low-freq-it}
Let $d=1$ or $d=2$, $0<R\leq r_1<r_2$ and $\frac18\leq s\leq 1$. Let $\psi_1(x)$ and $\psi_2(x)$ be smooth radial cutoff functions such that
$$\psi_1(x) =
\begin{cases}
0 & \text{on } |x|\leq r_1 \\
1 & \text{on }|x| \geq \frac12(r_1+r_2)
\end{cases}
\qquad
 \psi_2(x) =
\begin{cases}
0 & \text{on } |x|\leq \frac12(r_1+r_2) \\
1 & \text{on }|x| \geq r_2.
\end{cases}
$$
Then
$$
\| D_x^s \psi_2 u \|_{L_{[0,1)}^\infty L_x^2} \lesssim  1 + \| \la D_x\ra ^{s-\frac18} \psi_2 u \|_{L_{[0,1)}^\infty L_x^2} \,.
$$
\end{lemma}
\begin{proof}
Let $\chi(\rho)$ be a smooth function such that $\chi(\rho)=1$ for $|\rho|\leq 1$ for $\chi(\rho)=0$ for $|\rho|\geq 2$.  Let $P_-=P_{\leq (T-t)^{-3/4}}$ be the time-dependent multiplier operator defined by $\widehat{Pf}(\xi) = \chi((T-t)^{3/4}|\xi|) \hat f(\xi)$ (where the Fourier transform is in space only).  Note that the Fourier support of $P$ at time $t_k=1-2^{-k}$ is $\lesssim 2^{3k/4}$.  We further have that
$$
\partial_t P_- f = \frac34i(1-t)^{-1/4}QD_x f + P\partial_t f \,,
$$
where $Q=Q_{(1-t)^{-3/4}}$ is the time-dependent multiplier
$$
\widehat{Qf}(\xi) = \chi'((1-t)^{3/4}|\xi|) \hat f(\xi) \,.
$$
Note that the Fourier support of $Q$ at time $t_k=1-2^{-k}$ is $\sim 2^{3k/4}$.  Note also that if $g=g(x)$ is any function, then
\begin{equation}
\label{E:exchange-2}
\| PD_x^\alpha g \|_{L_x^2} \leq (1-t)^{-3\alpha/4} \|g\|_{L_x^2} \,.
\end{equation}
Let $w=P_-\psi_2 u$.  Taking $\tilde \psi_2 = \nabla_x \psi_2$ and $\tilde{\tilde \psi}_2 = \Delta_x \psi_2$, we have
\begin{align*}
i\partial_t w + \Delta w &=
\begin{aligned}[t]
&-i (1-t)^{-1/4} Q \cdot \nabla_x \; w - P_- \psi_2 |u|^{4/d} u \\
&+ 2 P_- \nabla_x \cdot[\tilde \psi_2 u]  - P_- \tilde{\tilde\psi}_2 u
\end{aligned}\\
&= F_1+F_2+F_3+F_4\,.
\end{align*}
By the energy method,
$$
\|D_x^s w\|_{L_{[0,1)}^\infty L_x^2}^2 \lesssim \|D_x^s w(0)\|_{L_x^2}^2 + \int_0^1 | \la D_x^s F_1(s), D_x^s w(s) \ra_{L_x^2} | \, ds + 10\sum_{j=2}^4 \| D_x^s F_j \|_{L_{[0,1)}^1 L_x^2}^2 \,.
$$
For term $F_1$, we argue as follows.  Let $\tilde Q$ be a projection onto frequencies of size $(1-t)^{-3/4}$.  Then
$$
\int_0^1 | \la D_x^s F_1(s), D_x^s w(s) \ra_{L_x^2} | \, ds \lesssim \int_0^1 (1-s)^{-1/4} \| D_x^{\frac12+s} \tilde Q \psi_2 u(s) \|_{L_x^2}^2 \, ds\,.
$$
Applying \eqref{E:exchange-2} with $\alpha=\frac12$, we can control the above by
$$
\int_0^1 (1-s)^{-1} \| D_x^s \tilde Q \psi_2 u(s) \|_{L_x^2}^2 \, ds.
$$
Dividing the time interval $[0,1) = \cup_{k=1}^\infty [t_k, t_{k+1})$, we bound the above by
$$
\sum_{k=1}^{+\infty} 2^k \int_{t_k}^{t_{k+1}} \| D_x^s P_{2^{3k/4}} \psi_2 u(s) \|_{L_x^2}^2 \, ds \lesssim \sum_{k=1}^{+\infty} \| D_x^s P_{2^{3k/4}} \psi_2 u(s) \|_{L_{[t_k,t_{k+1})}^\infty L_x^2}^2,
$$
where $P_{2^{3k/4}}$ is the projection onto frequencies of size $\sim 2^{3k/4}$ (and not $\lesssim 2^{3k/4}$).  However, writing $u(t)=e^{it\Delta}u_0 + (u(t)-e^{it\Delta}u_0)$, the above is controlled by (taking $s=1$, the worst case)
$$
\sum_{k=1}^\infty \|\nabla_x P_{2^{3k/4}} u_0 \|_{L_x^2}^2 + \sum_{k=1}^{+\infty} \|\nabla_x P_{2^{3k/4}} (u(t)-e^{it\Delta}u_0) \|_{L_x^2}^2\,.
$$
By Prop. \ref{P:high-freq}, \eqref{E:step-40},
$$
\| \nabla_x u_0 \|_{L_x^2}^2 + \sum_{k=1}^{+\infty} 2^{-k/8} \lesssim 1 \,.
$$
In conclusion for term $F_1$ we obtain
$$
\int_0^1 | \la D_x^s F_1(s), D_x^s w(s) \ra_{L_x^2} | \, ds  \lesssim 1 \,.
$$

We next address term $F_2$.  Insert $\psi_2 \psi_1^{\frac4d+1} = \psi_2$, then apply \eqref{E:exchange-2} with $\alpha=s$ to obtain (in the worst case $s=1$),
$$
\| D_x^s F_2 \|_{L_{[0,1)}^1 L_x^2} \lesssim \| (1-t)^{-3/4}\psi_2 | u|^{4/d} u \|_{L_{[0,1)}^1 L_x^2} \lesssim \| (1-t)^{-3/4} \|\psi_1 u \|_{L_x^{2(\frac{4}{d}+1)}}^{\frac{4}{d}+1} \|_{L_{[0,1)}^1} \,.
$$
We consider the cases $d=1$ and $d=2$ separately.  When $d=1$,
$$
\|\psi_1 u \|_{L_x^{10}} \lesssim \|D_x^{2/5} \psi_1 u \|_{L_x^2} \lesssim 1 \,,
$$
by Lemma \ref{L:a-little-1}.  Consequently,
$$
\|D_x^s F_2 \|_{L_{[0,1)}^1 L_x^2} \lesssim \|(1-t)^{-3/4}\|_{L_{[0,1)}^1} \lesssim 1\,.
$$
On the other hand, when $d=2$, we have
$$
\| \psi_1 u \|_{L_x^6} \lesssim \|D_x^{2/3} \psi_1 u \|_{L_x^2} \lesssim \|D_x^{1/2} \psi_1 u \|_{L_x^2}^{2/3} \|\nabla_x \psi_1 u \|_{L_x^2}^{1/3} \lesssim (1-t)^{-1/6}
$$
by Lemma \ref{L:a-little-2} and \eqref{E:remainder-bounds2}.  Consequently,
$$
\| D_x^s F_2 \|_{L_{[0,1)}^1 L_x^2} \lesssim \| (1-t)^{-3/4} (1-t)^{-1/6} \|_{L_{[0,1)}^1} \lesssim 1 \,.
$$
Next, we address term $F_3$.  By \eqref{E:exchange-2} with $\alpha=\frac98$,
$$
\|D_x^s F_3 \|_{L_{[0,1)}^1L_x^2} \lesssim \|(1-t)^{-27/32} \|_{L_{[0,1)}^1} \|D_x^{s-\frac18} (\tilde \psi_2 u ) \|_{L_{[0,1)}^\infty L_x^2} \,.
$$
Since $\|(1-t)^{-27/32} \|_{L_{[0,1)}^1} \sim 1$ and the support of $\tilde \psi_2$ is contained in the set where $\psi_1=1$, we have
$$
\|D_x^s F_3 \|_{L_{[0,1)}^1L_x^2} \lesssim  \|\la D_x\ra^{s-\frac18} \psi_1 u \|_{L_{[0,1)}^\infty L_x^2} \,.
$$
Finally, we consider $F_4$.  We have
\begin{align*}
\|D_x^s F_4 \|_{L_{[0,1)}^1 L_x^2}
&\lesssim \| \la \nabla_x \ra P_- \psi_1 u \|_{L_{[0,1)}^1 L_x^2} \\
&\lesssim \| (1-t)^{-3/4} \|_{L_{[T_1,1)}^1} \|u\|_{L_{[0,1)}^\infty L_x^2} \\
&\lesssim 1
\end{align*}
by \eqref{E:exchange-2} with $\alpha=1$.
\end{proof}

The next proposition completes the proof of  Theorem \ref{T:main}.

\begin{proposition}
\label{P:finite-speed-new}
Suppose that $u(t)$ solving \eqref{E:NLS} with $H^1$ initial data satisfies \eqref{E:log-log-rescaled}.  Fix $R>0$.  Then
$$
\| u\|_{L_{[0,1)}^\infty H_{|x|\geq R}^1} \lesssim 1 \,.
$$
\end{proposition}
\begin{proof}
Iterate Lemma \ref{L:low-freq-it} eight times on successively larger external regions.
\end{proof}

Prop. \ref{P:finite-speed-new} completes the proof of Theorem \ref{T:main}.

\section{Application to 3d standing sphere blow-up}
\label{S:3d-sphere}

We now outline the proof of Theorem \ref{T:main-2} utilizing the techniques of \S\ref{S:high-frequency}--\ref{S:finite-speed}.  Theorem \ref{T:main-2} pertains to radial solutions of \eqref{E:NLS-5}.
We define the initial data set $\mathcal{P}$ as in\footnote{We are considering the case dimension $d=3$ (in their notation $N=3$).} Rapha\"el-Szeftel \cite{RS}, Def. 1 on p. 980-981, except that condition (v) is replaced by $\|\tilde u_0 \|_{H^1(|r-1|\geq \frac1{10})} \leq \epsilon^5$.    The goal then becomes to complete the proof of the bootstrap Prop. 1 on p. 982, where the ``improved regularity estimates'' (35)-(36)-(37) are effectively replaced with
$$
\|u(t)\|_{L_{[0,t_1]}^\infty H_{|x|\leq \frac12}^1} \leq \epsilon \,.
$$

Let us formulate a more precise statement:

\begin{proposition}[partial bootstrap argument]
\label{P:bootstrap}
Let $Q$ be the 1d ground state given by \eqref{E:ground-state-1d}, and let $\epsilon>0$, $T>0$ be fixed with $T\leq \epsilon^{200}$.  Suppose that $u(t)$ is a radial 3d solution to
$$
i\partial_t u + \Delta u + |u|^4 u=0
$$
on an interval $[0,T'] \subset [0,T)$ such that the following ``bootstrap inputs'' hold:
\begin{enumerate}
\item
\label{BSI-1}
There exist parameters $\lambda(t)>0$,
$\gamma(t)\in \mathbb{R}$, and $|r(t)-1|\leq \frac1{10}$, such that if we define
\begin{equation}
\label{E:tilde-u-def}
\tilde u(r,t) = u(r,t) - \frac{1}{\lambda(t)^{1/2}} Q\left( \frac{r-r(t)}{\lambda(t)} \right),
\end{equation}
then, for $0\leq t \leq T'$,
\begin{equation}
\label{E:step-82}
\|\nabla u(t)\|_{L_x^2} = \lambda(t)^{-1} \sim \left( \frac{\log|\log(T-t)|}{T-t} \right)^{1/2} \,,
\end{equation}
and
\begin{equation}
\label{E:bs-error-bd}
\|\nabla \tilde u(t) \|_{L_x^2} \lesssim \frac{1}{|\log(T-t)|^{1+}(T-t)^{1/2}} \,.
\end{equation}
\item
\label{BSI-3} Interior Strichartz control:
$\| \la \nabla \ra u(t) \|_{L_{[0,T']}^5 L_{|x|\leq \frac12}^{30/11}} \leq \epsilon$.
\item
\label{BSI-4} Initial data remainder control:
$\|\la \nabla \ra \tilde u_0\|_{L_x^2} \leq \epsilon^5$.
\end{enumerate}
Then we have the following ``bootstrap output''
\begin{equation}
\label{BSO}
\| \la \nabla \ra u(t) \|_{L_{[0,T']}^\infty L_{|x|\leq \frac12}^2} + \| \la \nabla \ra u(t) \|_{L_{[0,T']}^5 L_{|x|\leq \frac12}^{30/11}} \lesssim \epsilon^5 \,.
\end{equation}
\end{proposition}

The goal of this section is to prove Prop. \ref{P:bootstrap}, which shows that the bootstrap input \eqref{BSI-3} is reinforced.  Prop. \ref{P:bootstrap} is, however, an incomplete bootstrap and by itself does not establish Theorem \ref{T:main-2}.  The analysis which uses \eqref{BSO} to reinforce the bootstrap assumption \eqref{BSI-1} is rather elaborate but will be omitted here as it follows the arguments in Rapha\"el \cite{Rap} and Rapha\"el-Szeftel \cite{RS}.  Moreover, these papers demonstrate how the assertions in Theorem \ref{T:main-2} follow.

The proof of Prop. \ref{P:bootstrap} follows the methods developed in \S\ref{S:high-frequency}--\ref{S:finite-speed} used to prove Theorem \ref{T:main}.  We do not, however, rescale the solution so that $T=1$ as was done in \S\ref{S:high-frequency}.

\begin{remark}
\label{R:bs-notation}
Let us list some notational conventions for the rest of the section.
We take $t_k=T-2^{-k}$ and denote $I_k=[0,t_k]$.  Let $v(r,t)=ru(r,t)$, and consider $v$ as a 1d function in $r$ extended to $r<0$ as an odd function.  Note that $v$ solves
$$
i\partial_t v + \partial_r^2 v = -r^{-4}|v|^4v \,.
$$
The frequency projection $P_N$ will always refer to the 1d frequency projection in the $r$-variable.  The Bourgain norm $\|v\|_{X_{s,b}}$ refers to the 1d norm in the $r$-variable.

Let $\lambda_0=\lambda(0)$ and take $k_0\in \mathbb{N}$ such that $2^{-k_0/2}(\log k_0)^{-1/2} \sim \lambda_0$.  We then have $T\sim 2^{-k_0}$.  The assumption $T\leq \epsilon^{40}$ equates to $2^{-k_0/8} \leq \epsilon^5$.  Note that $\lambda(t_k) = 2^{-k/2}(\log k)^{-1/2}$.
\end{remark}

\begin{lemma}[smallness of initial--data]
\label{L:initial-data}
Under the assumption \eqref{BSI-4} in Prop. \ref{P:bootstrap} on the initial data, and with $v_0=ru_0$, we have
$$
\|P_{\geq 2^{3k_0/4}} \partial_r v_0\|_{L_r^2} + \| \partial_r v_0 \|_{L_{r\leq \frac12}^2} \lesssim \epsilon^5 \,.
$$
\end{lemma}
\begin{proof}
Let $\tilde v_0 = r\tilde u_0$.
Since $\partial_r \tilde v_0 = \tilde u_0 + r\partial_r \tilde u_0$, we have by Hardy's inequality
\begin{align*}
\| \partial_r \tilde v_0 \|_{L^2_r}
&\lesssim \| |x|^{-1} \tilde u_0 \|_{L_x^2} + \|\nabla \tilde u_0 \|_{L_x^2} \\
&\lesssim \|\nabla \tilde u_0\|_{L^2_x} \\
&\lesssim \epsilon^5 \,.
\end{align*}
Recalling the definition of $\tilde u_0=\tilde u(0)$ in \eqref{E:tilde-u-def} (with $t=0$), we have
$$
v_0 = \frac{r}{\lambda_0^{1/2}} Q\left( \frac{r-r_0}{\lambda_0} \right) + \tilde v_0\,.
$$
The result then follows from the exponential localization and smoothness of $Q$.
\end{proof}

\begin{lemma}[radial Strichartz]
\label{L:radial-Strichartz}
Suppose that $u(t)$ is a 3d radial solution to
$$
i\partial_t u + \Delta u = f.
$$
Let $v(r,t) = ru(r,t)$ and $g(r,t)=rf(r,t)$ and consider $v$ as a 1d function in $r$ (extended to be odd), so that
$$
i\partial_t v + \partial_r^2 v = g\,.
$$
Then for $(q,r)$ and $(\tilde q,\tilde r)$ satisfying the 3d admissibility condition,
$$
\| r^{\frac{2}{p}-1} v \|_{L_t^qL_r^p} \lesssim \|v_0\|_{L_r^2} + \|r^{\frac{2}{p'}-1}g\|_{L_t^{\tilde q'}L_r^{\tilde p'}} \,.
$$
\end{lemma}
\begin{proof}
The left-hand side is equivalent to $\| \nabla u \|_{L_t^qL_x^p}$ and the right-hand side is equivalent to $\|u_0\|_{L_x^2} + \|f\|_{L_t^{\tilde q'}L_x^{\tilde p}}$, so it is just a restatement of the 3d Strichartz estimates.
\end{proof}

\begin{lemma}[3d -- 1d conversion]
\label{L:conversion}
Suppose that $u(x)$ is a 3d radial function, and write $u(r)= u(x)$.  Let $v(r)=ru(r)$.
Then for $1< p < 3$, we have
\begin{equation}
\label{E:step-60}
\|r^{\frac{2}{p}-1} \partial_r v \|_{L^p_r} \lesssim  \|\nabla_x u \|_{L_x^p} \,.
\end{equation}
Also for $\frac{3}{2}< p< +\infty$, we have
\begin{equation}
\label{E:step-61}
\|\nabla_x u \|_{L_x^p}  \lesssim \| r^{\frac{2}{p}-1} \partial_r v\|_{L_r^p} \,.
\end{equation}
Consequently, for 3d admissible pairs $(q,p)$ such that $2\leq p<3$, we have
\begin{equation}
\label{E:step-62}
\| \nabla u \|_{L_t^qL_x^p} \sim \|r^{\frac{2}{p}-1} \partial_r v \|_{L_t^qL_r^p}.
\end{equation}
\end{lemma}
We remark that $q=5$, $p=\frac{30}{11}$ falls within the range of validity for \eqref{E:step-62}.

\begin{proof}
The proof of \eqref{E:step-60} and \eqref{E:step-61} is a standard application of the Hardy inequality.

First, we prove \eqref{E:step-60}.  Using $v=ru$,
$$
r^{\frac{2}{p}-1}\partial_r v = r^\frac{2}{p}\partial_r u + r^{\frac{2}{p}-1} u,
$$
and thus,
$$
\| r^{\frac{2}{p}-1} \partial_r v \|_{L_r^p} \leq \| r^\frac{2}{p}\partial_r u \|_{L_r^p} + \|r^{\frac{2}{p}-1} u \|_{L_r^p} \,.
$$
We have, for $r>0$,
$$
u(r) = -(u(+\infty)-u(r)) = \int_{s=1}^{+\infty} \frac{d}{ds}[ u(sr)] \, ds = \int_{s=1}^{+\infty} u'(sr) r \, ds.
$$
By the Minkowski integral inequality,
$$
\| r^{\frac{2}{p}-1} u \|_{L_r^p} \leq \int_{s=1}^{+\infty} \| u'(sr) r^{\frac{2}{p}} \|_{L_{r>0}^p} \, ds.
$$
Changing variable $r\mapsto s^{-1}r$, we obtain that the right-hand side is bounded by
$$
\left( \int_{s=1}^{+\infty} s^{-\frac{3}{p}} \, ds \right) \|r^{\frac{2}{p}} u' \|_{L_{r>0}^p}
$$
and the $s$-integral is finite provided $p<3$.

Next, we prove \eqref{E:step-61}.  We have
$$
r^{\frac{2}{p}} \partial_r u = r^{\frac{2}{p}} \partial_r (r^{-1}v) = -r^{\frac{2}{p}-2}v + r^{\frac{2}{p}-1}\partial_r v \,,
$$
and hence,
$$
\| r^{\frac{2}{p}} \partial_r u \|_{L_r^p} \leq \| r^{\frac{2}{p}-2} v\|_{L_r^p} + \| r^{\frac{2}{p}-1}\partial_r v \|_{L_r^p} \,.
$$
We have
$$
v(r) = v(r)-v(0) = \int_{s=0}^1 \frac{d}{ds}[ v(sr)] \, ds = \int_{s=0}^1 v'(sr) r \, ds.
$$
By the Minkowski integral inequality,
$$
\|r^{\frac{2}{p}-2} v\|_{L_r^p} \leq \int_{s=0}^1 \|v'(sr) r^{\frac{2}{p}-1} \|_{L_r^p} \, ds.
$$
Changing variable $r\mapsto s^{-1}r$ in the right-hand side, we obtain
$$
\|r^{\frac{2}{p}-2} v\|_{L_r^p} \leq \left(\int_{s=0}^1 s^{-\frac{3}{p}+1} ds \right) \|v'(r) r^{\frac{2}{p}-1} \|_{L_r^p}
$$
and the $s$-integral is finite provided $p>\frac{3}{2}$.
\end{proof}

The replacement for Lemma \ref{L:local} is Lemma \ref{L:ring-local} below.  Notice that the difference is that in Lemma \ref{L:ring-local}, we only use $b<\frac12$ when working at $\dot H^1$ regularity.

\begin{lemma}
\label{L:ring-local}
Suppose that the assumptions of Prop. \ref{P:bootstrap} and Remark \ref{R:bs-notation} hold.  Then for $\frac12-\delta\leq b<\frac12$,
\begin{equation}
\label{E:step-80}
\| \partial_r v \|_{X_{0,b}(I_k)} \lesssim 2^{kb}(\log k)^{b+\frac12} = (T-t)^{-b}(\log|\log(T-t)|)^{b+\frac12} \,.
\end{equation}
Also, for $\frac12-\delta<b<\frac12+\delta$,
\begin{equation}
\label{E:step-81}
\| v \|_{X_{0,b}(I_k)} \lesssim_\delta 2^{k\delta} = (T-t)^{-\delta} \,.
\end{equation}
\end{lemma}
\begin{proof}
We will only carry out the proof of \eqref{E:step-80}, which stems from \eqref{E:step-82}.\footnote{The need to take $b<\frac12$ comes from Lemma \ref{L:flow}, \eqref{E:step-101} versus \eqref{E:step-102}; when working at $\dot H^1$ regularity near the origin, we cannot suffer any loss of derivatives.  The fact that $\| \partial_r v \|_{X_{0,b}(I_k)}$ for $b<\frac12$ is only a $\dot H^1$ subcritical quantity is of no harm as the only application of \eqref{E:step-80} in the subsequent arguments is to control the solution for $r\geq \frac12$, where the equation is effectively $L^2$ critical.}
The proof of \eqref{E:step-81} is similar, and stems from the bound on $\|u(t)\|_{H^\delta}$ obtained from interpolation between \eqref{E:step-82} and mass conservation.

In the proof below, $T$ has no relation to the $T$ representing blow-up time in the rest of the article.

Let $\lambda=\lambda(t_k) = 2^{-k/2}(\log k)^{-1/2}$.
Let $r=\lambda R$, $x=\lambda X$, $t=\lambda^2T+t_k$.  Define the functions
$$
V(R,T) = \lambda^{1/2}v(\lambda R, \lambda^2T+t_k) = \lambda^{1/2}v(r,t) \,,
$$
$$
U(X,T) = \lambda^{1/2}u(\lambda X, \lambda^2T+t_k) = \lambda^{1/2}u(x,t) \,.
$$
Note that the identity $v(r)=ru(r)$ corresponds to $V(R)=\lambda RU(R)$.

We study $V(R,T)$ on $T\in [0,\log k]$, which corresponds to $t\in [t_k,t_{k+1}]$.  We have $\|V\|_{L_R^2} = \|v\|_{L_r^2} \sim O(1)$ (by mass conservation) and $\|\partial_R V\|_{L_R^2} =\lambda \|\partial_r v\|_{L_r^2}$.  Hence, $\|\partial_R V\|_{L_{[0,\log k]}^\infty L^2_R} = O(1)$.  The equation satisfied by $V$ is
$$
i\partial_T V + \partial_R^2 V = - \lambda^{-4}R^{-4}|V|^4V \,.
$$
Let $J=[a,b]$ be a unit-sized time interval in $[0,\log k]$.  Then by Lemma \ref{L:flow},
$$
\|\partial_R V \|_{X_{0,b}(J)} \lesssim \|\partial_R V(a)\|_{L^2} + \| \partial_R (\lambda^{-4} R^{-4} |V|^4V) \|_{L_J^1L_R^2} \,.
$$
Let $\chi_1(r)=1$ for $r\leq \frac14$ and $\supp \chi_1 \subset B(0,\frac38)$.  Let $\chi_2=1-\chi_1$.
Let $g_1 = \partial_R (\lambda^{-4}R^{-4} \chi_1(\lambda R) |V|^4V)$ and $g_2 = \partial_R(\lambda^{-4}R^{-4}\chi_2(\lambda R) |V|^4V)$, so that the above becomes
\begin{equation}
\label{E:step-88}
\|\partial_R V \|_{X_{0,b}(J)} \lesssim \|\partial_R V(a)\|_{L^2} + \|g_1\|_{L_J^1L_R^2} + \|g_2\|_{L_J^1L_R^2} \,.
\end{equation}
We begin with estimating $\|g_2\|_{L_J^1L_R^2}$.
We have
\begin{equation}
\label{E:step-83}
\|g_2\|_{L_J^1L_R^2} \lesssim \| V^5 \|_{L_J^1L_R^2} + \| V^4 (\partial_R V)\|_{L_J^1L_R^2}.
\end{equation}
We now treat the first term in \eqref{E:step-83}.  Of course, $\|V^5\|_{L_J^1L_R^2} = \|V\|_{L_J^5 L_R^{10}}^5$.  By Sobolev embedding $\|V\|_{L_R^{10}} \lesssim \|D_R^{2/5} V \|_{L_R^2}$ and by H\"older,
\begin{align*}
\|V\|_{L_J^5L_R^{10}} \lesssim |J|^{1/10} \| D_R^{2/5}V\|_{L_J^{10}L_R^2}
& \lesssim |J|^{1/10}( \|V\|_{L_J^{10} L_R^2} + \|\partial_R V\|_{L_J^{10}L_R^2})\\
& \leq |J|^{1/10}( |J|^{1/10} \, \|V\|_{L_J^\infty L_R^2} + \|\partial_R V\|_{L_J^{10}L_R^2}).
\end{align*}
Using that $\|V\|_{L_J^\infty L_R^2} \sim 1$, $|J| \sim 1$ and Lemma \ref{L:interp-1}, provided $\frac{2}{5}<b<\frac12$, we have
\begin{equation}
\label{E:step-84}
\|V\|_{L_J^5L_R^{10}} \lesssim |J|^{1/10}( 1 + \|\partial_R V \|_{X_{0,b}}).
\end{equation}
We now treat the second term in \eqref{E:step-83}, similarly estimating the term $\|V\|_{L_R^{10}}$.  We have
\begin{align*}
\| V^4 \partial_R V \|_{L_J^1L_R^2}
&\lesssim |J|^{7/20} \|V\|_{L_J^{10}L_R^{10}}^4 \|\partial_R V\|_{L_J^4L_R^{10}} \\
&\lesssim |J|^{7/20} (1+ \|\partial_R V \|_{L_J^{10}L_R^2})^4 \| \partial_R V\|_{L_J^4L_R^{10}}.
\end{align*}
Appealing to Lemma \ref{L:interp-1}, provided $\frac{9}{20}<b<\frac12$, we obtain
\begin{equation}
\label{E:step-85}
\|V^4 \partial_R V \|_{L_J^1L_R^2} \lesssim |J|^{7/20} (1+\|\partial_R V\|_{X_{0,b}})^5 \,.
\end{equation}
Combining \eqref{E:step-84} and \eqref{E:step-85}, we have
\begin{equation}
\label{E:step-86}
\|g_2\|_{L_J^1L_R^2} \lesssim |J|^{7/20} (1+\|\partial_R V\|_{X_{0,b}})^5 \,.
\end{equation}
Next we estimate $\|g_1\|_{L_J^1L_R^2}$. By rescaling,
$$
\| g_1 \|_{L_J^1L_R^2} = \lambda \| \partial_r (\chi_1 r^{-4} |v|^4v)\|_{L_{[t_k,t_{k+1}]}^1L_r^2}.
$$
Let $w = \tilde \chi_1 u$, where $\tilde \chi_1 =1$ on $\supp\chi_1$ but $\supp\tilde \chi_1 \subset B(0,\frac12)$.  Replacing $u=r^{-1}v$, we obtain $\partial_r ( r \chi_1 u^5) = \partial_r ( r \chi_1 w^5)$, and hence,
\begin{equation}
\label{E:step-89}
\begin{aligned}
\| g_1 \|_{L_R^2}
&\lesssim \lambda(\| w \|_{L_r^{10}}^5 + \|r w^4  \partial_r w \|_{L_r^2}) \\
&\lesssim \lambda(\| |x|^{-1/5}w \|_{L_x^{10}}^5 + \|w^4 \nabla w \|_{L_x^2}).
\end{aligned}
\end{equation}
By Hardy's inequality and 3d Sobolev embedding,
$$
\| |x|^{-1/5} w \|_{L_x^{10}} \lesssim \| D_x^{1/5} w \|_{L_x^{10}} \lesssim \| \nabla w \|_{L_x^{30/11}} \,.
$$
By H\"older's inequality and 3d Sobolev embedding,
$$
\|w^4 \nabla w \|_{L_x^2}  \leq \|w\|_{L_x^{30}}^4 \|\nabla w\|_{L_x^{30/11}} \lesssim \|\nabla w\|_{L_x^{30/11}}^5 \,.
$$
Returning to \eqref{E:step-89} and invoking \eqref{BSI-3} of Prop. \ref{P:bootstrap},
\begin{equation}
\label{E:step-87}
\|g_1\|_{L_{I_k}^1 L_r^2} \lesssim \lambda \| \nabla w \|_{L_{I_k}^5L_x^{30/11}}^5 \lesssim \lambda \epsilon^5 \,.
\end{equation}
By putting \eqref{E:step-86} and \eqref{E:step-87} into \eqref{E:step-88}, we obtain
$$
\| \partial_R V \|_{X_{0,b}(J)} \lesssim \|\partial_R V(a) \|_{L^2} + |J|^{7/20} ( 1+ \| \partial_R V\|_{X_{0,b}(J)})^5 + \lambda \epsilon^5 \,.
$$
From this, we conclude that we can take $|J|$ sufficiently small (but still ``unit-sized'' \footnote{meaning: with size independent of any small parameters like $\epsilon$ or $\lambda$}) so that it follows that
$$
\|\partial_R V \|_{X_{0,b}(J)} \leq O(1) \,.
$$
Square summing over unit-sized intervals $J$ filling $[0,\log k]$,
$$
\|\partial_R V \|_{X_{0,b}([0,\log k])} \lesssim (\log k)^{1/2} \,.
$$
This estimate scales back to
$$
\| \partial_r v \|_{X_{0,b}([t_k,t_{k+1}])} \lesssim (\log k)^{1/2}\lambda(t_k)^{-2b} = 2^{kb}(\log k)^{b+\frac12} \,.
$$
Now square sum over $k$ from $k=0$ to $k=K$ to obtain a bound of $2^{Kb}(\log K)^{b+\frac12}$ over the time interval $I_K$, which is the claimed estimate \eqref{E:step-80}.
\end{proof}

The analogue of Lemma \ref{L:iterator} will be Lemma \ref{L:iterator-ring} below.   We note that as a consequence of Lemma \ref{L:ring-local}, the hypothesis of Lemma \ref{L:iterator-ring} below is satisfied with $\alpha(k,N) = 2^{-k/2}N^{-1}$.

\begin{lemma}[high-frequency recurrence]
\label{L:iterator-ring}
Suppose that the assumptions of Prop. \ref{P:bootstrap} and Remark \ref{R:bs-notation} hold.  Let\footnote{Note the inclusion of one derivative in the definition of $\beta$, in contrast to the choice of definition for $\alpha$ in \S\ref{P:high-freq}.}
$$
\beta(k,N) \defeq
\|P_{\geq N} \partial_r v\|_{X_{0,\frac12-}(I_k)} \,.
$$
Then there exists an absolute constant $0<\mu  \ll 1$ such that for $N \geq 2^{k(1+\delta)/2}$, we have
\begin{equation}
\label{E:step-93}
\begin{aligned}
\indentalign \beta(k,N) +\| r^{\frac{2}{p}-1} P_{\geq N} \partial_r v \|_{L_{I_k}^q L_r^p}  \\
&\lesssim  \|P_{\geq N} \partial_r v_0 \|_{L_r^2} + 2^{k(1+\delta)/2}  N^{-1+\delta}\beta(k, \mu N)+ N^{-1+\delta} 2^{k\delta} \beta(k,\mu N)^2 + 2^{-k\delta} + \epsilon^5
\end{aligned}
\end{equation}
for all 3d admissible $(q,p)$.
\end{lemma}
\begin{proof}
Note that $v$ solves
$$
i\partial_t v + \partial_r^2 v = -r|u|^4u = -r^{-4}|v|^4v \,.
$$
Let $\chi_1(r)$ be a smooth function such that $\chi_1(r)=1$ for $|r|\leq \frac14$ and $\chi_1$ is supported in $|r|\leq \frac38$.  Let $\chi_2=1-\chi_1$.  Apply $P_{\geq N}\partial_r$ to obtain
$$
(i\partial_t + \partial_r^2) P_{\geq N} \partial_r v = g_1+g_2,
$$
where
$$
g_j(r) = - P_{\geq N} \partial_r (\chi_j \, r^{-4}\, |v|^4v) \,,  \quad j=1,2 \,.
$$
Then by Lemma \ref{L:flow}\footnote{Note that we were able to obtain the $L_{I_k}^1 L_r^2$ right-hand side (without $\delta$ loss), because we took $b<\frac12$ in the Bourgain norm.} and Lemma \ref{L:radial-Strichartz},
$$
\|P_{\geq N} \partial_r v\|_{X_{0,\frac12-}(I_k)} +\| r^{\frac{2}{p}-1} P_{\geq N} \partial_r v \|_{L_{I_k}^q L_r^p}  \lesssim  \|P_{\geq N} \partial_r v_0 \|_{L_r^2} + \|g_1 \|_{L_{I_k}^1L_r^2} + \|g_2 \|_{L_{I_k}^1L_r^2} \,.
$$

The term $\|g_2 \|_{L_t^1L_r^2}$ is controlled in a manner similar to the analysis in the proof of Lemma \ref{L:iterator}.  For this term, $\chi_2 \, r^{-4}$ and $\partial_r(\chi_2 \, r^{-4})$ are smooth bounded functions, with all derivatives bounded.  By Lemma \ref{L:commutator},
\begin{equation}
\label{E:step-90}
\|g_2 \|_{L_r^2} \lesssim  \|P_{\geq N} \la \partial_r\ra v^5 \|_{L_r^2} + N^{-1} \|\la \partial_r \ra v^5\|_{L_r^2} \,.
\end{equation}
By an analysis similar to the proof of Lemma \ref{L:iterator}, utilizing the bounds in Lemma \ref{L:ring-local}, we obtain
\begin{equation}
\label{E:step-91}
\|P_{\geq N} \la \partial_r\ra v^5 \|_{L_{I_k}^1 L_r^2} \lesssim 2^{k(1+\delta)/2}N^{-1+\delta} \beta(k,\mu N) + N^{-1+\delta} 2^{k\delta} \beta(k,\mu N)^2\,.
\end{equation}
Also by the Strichartz estimates, as in the proof of Lemma \ref{L:ring-local} above,
\begin{equation}
\label{E:step-92}
\| \la \partial_r \ra v^5 \|_{L_{I_k}^1 L_r^2} \lesssim \|D^\delta v \|_{X_{0,b}}^4 \|\partial_R v \|_{X_{0,b}} \lesssim 2^{k(1+\delta)/2}.
\end{equation}
Inserting \eqref{E:step-91} and \eqref{E:step-92} into \eqref{E:step-90}, we obtain
\begin{equation}
\label{E:star-2}
\|g_2 \|_{L_{I_k}^1 L_r^2} \lesssim  2^{k(1+\delta)/2}N^{-1+\delta} \beta(k,\mu N) + N^{-1+\delta} 2^{k\delta} \beta(k,\mu N)^2 + N^{-1} 2^{k(1+\delta)/2}.
\end{equation}
The last term, $N^{-1}2^{k(1+\delta)/2}$, gives the contribution $2^{-k\delta}$ in \eqref{E:step-93} due to the restriction $N \geq 2^{k(1+\delta)/2}$ (different $\delta$'s).

Next we address $\|g_1\|_{L_{I_k}^1L_r^2}$.  We estimate away $P_{\geq N}$
\begin{equation}
\label{E:step-70}
\|g_1\|_{L_{I_k}^1 L_r^2} \lesssim \|\tilde g_1 \|_{L_{I_k}^1L_r^2} \,,
\end{equation}
where (ignoring complex conjugates)
$$
\tilde g_1 = \partial_r (r^{-4}\chi_1 v^5).
$$
Let $w = \tilde \chi_1 u$, where $\tilde \chi_1 =1$ on $\supp\chi_1$ but $\supp\tilde \chi_1 \subset B(0,\frac12)$.  Replacing $u=r^{-1}v$, we obtain $\tilde g_1 = \partial_r ( r \chi_1 u^5) = \partial_r ( r \chi_1 w^5)$, and hence,
\begin{align*}
\|\tilde g_1 \|_{L_r^2}
&\lesssim \| w \|_{L_r^{10}}^5 + \|r w^4  \partial_r w \|_{L_r^2} \\
&\lesssim \| |x|^{-1/5}w \|_{L_x^{10}}^5 + \|w^4 \nabla w \|_{L_x^2}.
\end{align*}
By Hardy's inequality and 3d Sobolev embedding,
$$
\| |x|^{-1/5} w \|_{L_x^{10}} \lesssim \| D_x^{1/5} w \|_{L_x^{10}} \lesssim \| \nabla w \|_{L_x^{30/11}} \,.
$$
By H\"older's inequality and 3d Sobolev embedding,
$$
\|w^4 \nabla w \|_{L_x^2}  \leq \|w\|_{L_x^{30}}^4 \|\nabla w\|_{L_x^{30/11}} \lesssim \|\nabla w\|_{L_x^{30/11}}^5 \,.
$$
Hence,
$$
\|\tilde g_1 \|_{L_r^2} \lesssim \| \nabla w \|_{L_x^{30/11}}^5 \,.
$$
Returning to \eqref{E:step-70} and invoking \eqref{BSI-3} of Prop. \ref{P:bootstrap},
$$
\|g_1\|_{L_{I_k}^1 L_r^2} \lesssim \| \nabla w \|_{L_{I_k}^5L_x^{30/11}}^5 \lesssim \epsilon^5.
$$
\end{proof}

The analogue of Prop. \ref{P:high-freq} is

\begin{proposition}[high-frequency control]
\label{P:high-freq-ring}
Suppose that the assumptions of Prop. \ref{P:bootstrap} and Remark \ref{R:bs-notation} hold.
Then for any 3d Strichartz admissible pair $(q,p)$, we have
$$
\|P_{\geq 2^{3k/4}} \partial_r v \|_{X_{0,\frac12-}(I_k)} + \|r^{\frac{2}{p}-1}P_{\geq 2^{3k/4}} \partial_r v\|_{L_{I_k}^q L_r^p} \lesssim  \epsilon^5 \,.
$$
\end{proposition}
\begin{proof}
Several applications of Lemma \ref{L:iterator-ring}, just as Prop. \ref{P:high-freq} is deduced from Lemma \ref{L:iterator}.
\end{proof}

Due to the $\dot H^1$ criticality of the problem, we do not have improved regularity of $v(t)-e^{it\partial_r^2}v_0$ as was the case in Prop. \ref{P:high-freq}.  As a substitute, we can use the methods of Lemma \ref{L:iterator-ring} to obtain the following lemma:

\begin{lemma}[additional high-frequency control]
\label{L:extra-high-freq}
Suppose that the assumptions of Prop. \ref{P:bootstrap} and Remark \ref{R:bs-notation} hold.  Then
\begin{equation}
\label{E:star-1}
\left( \sum_{k=k_0}^{+\infty} \|P_{2^{3k/4}} \partial_r v \|_{L_{[t_{k-1},t_k]}^\infty L_r^2}^2 \right)^{1/2} \lesssim \epsilon^5 \,.
\end{equation}
\end{lemma}

\begin{proof}
It suffices to prove the estimate with the sum terminating at $k=K$, provided we obtain a bound independent of $K$.
For each $k$, $k_0\leq k \leq K$, write the integral equation on $I_k$.  For $t\in [t_{k-1},t_k]$
$$
v(t) = e^{it\partial_r^2}v_0  -i \int_0^t e^{i(t-t')\partial_r^2} (r^{-4} |v|^4v(t')) \, dt' \,.
$$
Apply $P_{2^{3k/4}}\partial_r $ to obtain
$$
P_{2^{3k/4}} \partial_r v(t) = P_{2^{3k/4}} e^{it\partial_r^2} \partial_r v_0 -i \int_0^t e^{i(t-t') \partial_r^2} P_{2^{3k/4}}\partial_r ( r^{-4}|v|^4v(t')) \, dt' \,.
$$
Estimate
$$
\|P_{2^{3k/4}} \partial_r v\|_{L_{[t_{k-1},t_k]}^\infty L_r^2} \leq \| P_{2^{3k/4}} \partial_r v_0 \|_{L_r^2} + \| P_{2^{3k/4}} \partial_r (r^{-4}|v|^4v) \|_{L_{I_k}^1 L_r^2} \,.
$$
By the inequality $(a + b)^2 \leq 2a^2 + 2b^2$, this implies
$$
\|P_{2^{3k/4}} \partial_r v\|_{L_{[t_{k-1},t_k]}^\infty L_r^2}^2 \lesssim \| P_{2^{3k/4}} \partial_r v_0 \|_{L_r^2}^2 + \| P_{2^{3k/4}} \partial_r (r^{-4}|v|^4v) \|_{L_{I_k}^1 L_r^2}^2 \,.
$$
Let $\chi_1(r)$ be a smooth function such that $\chi_1(r)=1$ for $|r|\leq \frac14$ and $\chi_1$ is supported in $|r|\leq \frac38$.  Let $\chi_2=1-\chi_1$.  Let
$$
g_j = P_{2^{3k/4}} \partial_r ( \chi_j r^{-4}|v|^4v) \,, \quad j=1,2 \,.
$$

Recall that in the proof of Lemma \ref{L:iterator-ring}, we showed that
$$
\|P_{\geq N} \partial_r \chi_2 r^{-4} |v|^4v \|_{L_{I_k}^1 L_r^2} \lesssim 2^{k(1+\delta)/2}N^{-1+\delta} \beta(k,\mu N) + N^{-1+\delta} 2^{k\delta} \beta(k,\mu N)^2 + N^{-1} 2^{k(1+\delta)/2} \,,
$$
and Prop. \ref{P:high-freq-ring} showed that $\beta(k,2^{3k/4}) \lesssim 1$.  Combining gives
$$
\|g_2 \|_{L_{I_k}^1L_r^2} \lesssim 2^{-k/8} \,,
$$
and hence,
$$
\left( \sum_{k=k_0}^K \|g_2 \|_{L_{I_k}^1L_r^2}^2 \right)^{1/2} \lesssim 2^{-k_0/8} \leq \epsilon^5\,.
$$

Now we address $g_1$.  Let $w=\tilde \chi_1 u$.  For each $k$, lengthen $I_k$ to $I\defeq I_K$ to obtain
$$
\sum_{k=k_0}^K \|g_1 \|_{L_{I_k}^1L_r^2}^2 \lesssim \| P_{2^{3k/4}} \partial_r (r^{-4}\chi_1 |w|^4w) \|_{\ell_k^2 L_I^1 L_r^2}^2\,.
$$
By the Minkowski inequality, for any space-time function $F$, we have
$$
\| P_{2^{3k/4}}  F \|_{\ell_k^2L_I^1 L_r^2} \leq \| P_{2^{3k/4}}F \|_{L_I^1 \ell_k^2 L_r^2} \lesssim \| F \|_{L_I^1 L_r^2}.
$$
Hence,
$$
\sum_{k=k_0}^K \|g_1 \|_{L_{I_k}^1L_r^2}^2 \lesssim \| \partial_r( \chi_1 r^{-4}|w|^4w) \|_{L_I^1 L_r^2}^2\, .
$$
At this point we proceed as in Lemma \ref{L:iterator-ring} to obtain a bound by $\epsilon^5$.
\end{proof}

Now we begin to insert spatial cutoffs away from the blow-up core and obtain the missing low frequency bounds.
The first step is to obtain a little regularity above $L^2$, since it is needed in the proof of Lemma \ref{L:low-freq-iterator}.

\begin{lemma}[small regularity gain]
\label{L:a-little}
Suppose that the assumptions of Prop. \ref{P:bootstrap} and Remark \ref{R:bs-notation} hold. Let $\psi_{3/4}(r)$ be a smooth function such that $\psi_{3/4}(r)=1$ for $|r|\leq \frac34$ and $\psi_{3/4}(r)=0$ for $|r|\geq \frac78$.  Then
$$
\| \la D_r\ra ^{3/7} \psi_{3/4} v \|_{L_{[0,T)}^\infty L_r^2} \lesssim \epsilon^5 \,.
$$
\end{lemma}
\begin{proof}  Taking $\psi=\psi_{3/4}$, let $w=\psi v$.  Then
\begin{align*}
i\partial_t w + \partial_r^2 w
&= \psi(i\partial_t + \partial_r^2)v + 2\partial_r (\psi'  v) - \psi'' v \\
&= - r^{-4}\psi |v|^4v + 2\partial_r (\psi'  v) - \psi'' v \\
&= F_1+F_2+F_3.
\end{align*}
Local smoothing and energy estimates provide the following estimate
\begin{equation}
\label{E:step-100}
\| D_r^{3/7} w \|_{L_{[0,T)}^\infty L_r^2} \lesssim
\begin{aligned}[t]
&\|D_r^{3/7} w_0 \|_{L_r^2} + \|D_r^{3/7} F_1 \|_{L_{[0,T)}^1L_r^2} \\
&+ \| D_r^{-1/2} D_r^{3/7} F_2 \|_{L_{[0,T)}^2L_r^2} + \| D_r^{3/7}F_3 \|_{L_{[0,T)}^1 L_r^2} .
\end{aligned}
\end{equation}
We begin with the $F_1$ estimate.  Let $\tilde \psi$ be a smooth function such that
$$
\tilde \psi(r) =
\left\{
\begin{aligned}
&0 && \text{if }r\leq \tfrac14 \\
&1 && \text{if }\tfrac12 \leq r \leq \tfrac78 \\
&0 && \text{if } r\geq \tfrac{7}{8}. %{15}{16} - since \psi is zero outside of 7/8
\end{aligned}
\right.
$$
Let $q=r^{-1}\tilde \psi v$.  By writing $1=(1-\tilde \psi^4)+\tilde \psi^4$, we obtain
$$
F_1 = -(1-\tilde \psi^4)\psi r^{-4} |v|^4 v - |q|^4 w .
$$
Note that $(1-\tilde \psi^4)\psi$ is supported in $|r|\leq \frac12$ and $\tilde \psi^4 \psi$ is supported in $\frac14\leq |r|\leq \frac{15}{16}$.

For the term $(1-\tilde \psi^4)\psi r^{-4} |v|^4 v$, we appeal to the bootstrap hypothesis \eqref{BSI-3} in the same way we did in the proof of Lemma \ref{L:iterator-ring} to obtain a bound by $\epsilon^5$.  As for the term $|q|^4w$, by the fractional Leibniz rule,
$$
\|D_r^{3/7}(|q|^4 w) \|_{L_{[0,T)}^1L_r^2} \lesssim \| D_r^{3/7} |q|^4\|_{L_{[0,T)}^1L_r^{7/3}} \|w\|_{L_{[0,T)}^\infty L_r^{14}} + \| |q|^4 \|_{L_{[0,T)}^1L_r^\infty} \|D_r^{3/7} w \|_{L_{[0,T)}^\infty L_r^2} \,.
$$
By Sobolev embedding and Gagliardo-Nirenberg,
$$
\| D_r^{3/7} |q|^4 \|_{L_r^{7/3}} + \| |q|^4 \|_{L_r^\infty} \lesssim \| q\|_{L_r^2}^2 \|\partial_r q \|_{L_r^2}^2 \,,
$$
$$
\|w\|_{L_r^{14}} \lesssim \|D_r^{3/7} w \|_{L_r^2} \,.
$$
Hence,
$$
\|D_r^{3/7}(|q|^4 w) \|_{L_{[0,T)}^1L_r^2} \lesssim \|q\|_{L_{[0,T)}^\infty L_r^2}^2 \| \partial_r q\|_{L_{[0,T)}^2 L_r^2}^2 \|D_r^{3/7} w\|_{L_{[0,T)}^\infty L_r^2} \,.
$$
By \eqref{E:bs-error-bd},  $\|\partial_r q \|_{L_{[0,T)}^2L_r^2}\lesssim (|\log T|)^{-1} \lesssim (\log \epsilon^{-1})^{-1}$.  Consequently, we obtain
$$
\|D_r^{3/7} F_1 \|_{L_{[0,T)}^1 L_r^2} \lesssim \epsilon^5 + (\log \epsilon^{-1})^{-1} \|D_r^{3/7} w \|_{L_{[0,T)}^\infty L_r^2} \,.
$$
As for $F_2$, we start by bounding
$$
\|D_r^{-1/2}D_r^{3/7} F_2 \|_{L_{[0,T)}^2 L_r^2} \lesssim \| D_r^{13/14} (\psi' \, v) \|_{L_{[0,T)}^2 L_r^2} \,.
$$
On the support of $\psi'$, we have $v=rq$. Noting that on the support of $\psi'$ we have $r \sim 1$ and using the interpolation, we get
$$
\|D_r^{13/14} (\psi' r q)\|_{L_r^2} \lesssim \|q \|_{L_r^2} + \|q\|_{L_r^2}^{1/14} \|\partial_r q\|_{L_r^2}^{13/14} \,.
$$
By \eqref{E:bs-error-bd},
$$
\| \|\partial_r q \|_{L_r^2}^{13/14} \|_{L_{[0,T)}^2} \lesssim T^{1/28} \lesssim \epsilon^5 \,.
$$
Consequently,
$$
\|D_r^{-1/2}D_r^{3/7} F_2 \|_{L_{[0,T)}^2 L_r^2} \lesssim T^{1/2}+T^{1/28} \lesssim \epsilon^5 \,.
$$
Finally, for the term $F_3$, we estimate
$$
\| D_r^{3/7} F_3 \|_{L_{[0,T)}^1 L_r^2} \lesssim  \|q\|_{L_{[0,T)}^1 L_r^2} + \|\partial_r q \|_{L_{[0,T)}^1L_r^2} \lesssim T+T^{1/2} \lesssim \epsilon^5 \,.
$$
Collecting the above estimates and inserting into \eqref{E:step-100}, we obtain
$$
\|D_r^{3/7} w \|_{L_{[0,T)}^2 L_r^2} \lesssim \|D_r^{3/7}w_0 \|_{L_r^2} + (\log \epsilon^{-1})^{-1} \|D_r^{3/7} w\|_{L_{[0,T)}^\infty L_r^2} + \epsilon^5 \,,
$$
and the result follows (by bootstrap assumption \eqref{BSI-4}, $\|D_r^{3/7} w_0 \|_{L_r^2} \lesssim \epsilon^5$).
\end{proof}

We will need to apply the following lemma eight times in the proof of Prop. \ref{P:ring-energy} below.  As in \S \ref{S:finite-speed}, the use of the frequency projection $P_{\lesssim (T-t)^{-3/4}}$ and the process of exchanging derivatives for time-factors via \eqref{E:exchange} is essentially an appeal to the finite speed of propagation for low frequencies.

\begin{lemma}[low frequency recurrence]
\label{L:low-freq-iterator}
Suppose that the assumptions of Prop. \ref{P:bootstrap} and Remark \ref{R:bs-notation} hold.   Let $\frac58 < r_1<r_2<\frac34$ and $\frac18 \leq s\leq 1$.  Let $\psi_1(r)$ and $\psi_2(r)$ be smooth cutoff functions such that
$$
\psi_1(r) = \begin{cases} 1 & \text{on }|r|\leq r_1 \\ 0 & \text{on }|r|\geq \frac12(r_1+r_2) \end{cases} \, \qquad
\psi_2(r) = \begin{cases} 1 & \text{on }|r|\leq \frac12(r_1+r_2) \\ 0 & \text{on }|r|\geq r_2 \end{cases} \,.
$$
Then
$$
\| D_r^s (\psi_1 v) \|_{L_{[0,T)}^\infty L_r^2} \lesssim \| D_r^{s-\frac18} (\psi_2 v) \|_{L_{[0,T)}^\infty L_r^2} + \epsilon^5 \,.
$$
\end{lemma}
\begin{proof}
Let $\chi(\xi)=1$ for $|\xi|\leq 1$ and $\chi(\xi)= 0$ for $|\xi|\geq 2$ be a smooth function.  Let $P=P_{\leq (T-t)^{-3/4}}$ be the time-dependent multiplier operator defined by $\widehat{Pf}(\xi) = \chi((T-t)^{3/4}\xi) \hat f(\xi)$ (where Fourier transform is in space only).  Note that the Fourier support of $P$ at time $T-t=2^{-k}$ is $\lesssim 2^{3k/4}$.  We further have that
$$
\partial_t Pf = \tfrac34i(T-t)^{-1/4} Q\partial_r f + P\partial_t f \,,
$$
where $Q=Q_{(T-t)^{-3/4}}$ is the time-dependent multiplier
$$
\widehat{Q h}(\xi) = \chi'((T-t)^{3/4}\xi) \, \widehat h(\xi) \,.
$$
Note that the Fourier support of $Q$ at time $t=T-2^{-k}$ is $\sim 2^{3k/4}$.  Note also that if $g=g(r)$ is any function, then
\begin{equation}
\label{E:exchange}
\| PD_r^\alpha g \|_{L_r^2} \leq (T-t)^{-3\alpha/4} \|  g\|_{L_r^2}.
\end{equation}
Let $\tilde \psi$ be a smooth function such that
$$
\tilde \psi(r) =
\begin{cases}
0 & \text{if }|r|\leq \frac14 \\
1 & \text{if }\frac12 \leq |r| \leq \frac12(r_1+r_2) \\
0 & \text{if }|r|\geq r_2.
\end{cases}
$$
Let $w= P_{\leq (T-t)^{-3/4}} D_r^s (\psi_1 v)$.  By Prop. \ref{P:high-freq-ring}, it suffices to show that  $\|w\|_{L_{[0,T)}^\infty L_r^2} \lesssim \| D_r^{s-\frac18} (\psi_2 v) \|_{L_{[0,T)}^\infty L_r^2} + \epsilon^5$.  Note that $w$ solves
\begin{align*}
i\partial_t w + \partial_r^2 w
&= - \tfrac34(T-t)^{-1/4} Q\partial_r D_r^s (\psi_1 v) - P D_r^s (\psi_1 r^{-4} |v|^4 v)\\
& + 2 P \partial_r D_r^s (\psi_1' v) - P D_r^s (\psi_1'' v)\\
& = F_1+F_2+F_3+F_4 \,.
\end{align*}
By the energy method, we obtain
$$
\|w\|_{L_t^\infty L_r^2}^2 \leq \|w_0\|_{L_r^2}^2 + \int_0^T|\la F_1, w \ra_{L_r^2}| + 10\sum_{j=2}^4 \|F_j \|_{L_{[0,T)}^1L_r^2}^2 \,.
$$
We estimate $F_1$ using Lemma \ref{L:extra-high-freq} as follows.\footnote{It seems that the energy method is needed here, since it furnishes $\int_0^T |\la F_1, w \ra_{L_r^2}|$; we cannot see a way to estimate $\|F_1\|_{L_{[0,T)}^1L_r^2}$.  Indeed, by pursuing the method here, one ends up with a bound $\|F_1\|_{L_{[0,T)}^1L_r^2} \lesssim \sum_{k=k_0}^\infty \|P_{2^{3k/4}} \psi_1 v \|_{L_r^2}$, which is not controlled by Lemma \ref{L:extra-high-freq}, since it is not a \emph{square} sum.}
Let $\tilde Q$ be a projection onto frequencies of size $\sim (T-t)^{-3/4}$ (importantly, \emph{not} $\lesssim (T-t)^{-3/4}$).  Then
$$
\int_0^T |\la F_1, w \ra_{L_r^2} | \lesssim \int_0^T (T-t)^{-1/4} \|\tilde Q D_r^{\frac12+s} (\psi_1 v) \|_{L_r^2}^2 \,.
$$
It suffices to take $s=1$, the worst case.  The presence of $\tilde Q$ allows for the exchange $D_r^{1/2} \sim (T-t)^{-3/8}$, which gives
$$
\int_0^T |\la F_1, w \ra_{L_r^2} | \lesssim \int_0^T (T-t)^{-1} \|\tilde Q \partial_r (\psi_1 v) \|_{L_r^2}^2.
$$
By decomposing $[0,T) = \cup_{k=k_0}^\infty [t_k,t_{k+1}]$, and using that on $[t_k,t_{k+1}]$, $(T-t)^{-1} = 2^k$, we have
$$
\int_0^T (T-t)^{-1} \|\tilde Q \partial_r (\psi_1 v) \|_{L_r^2}^2 = \sum_{k=k_0}^\infty \int_{[t_k,t_{k+1}]} 2^k \|P_{2^{3k/4}} \partial_r (\psi_1 v)\|_{L_r^2}^2 \,.
$$
Since $|[t_k,t_{k+1}]|=2^{-k}$, the above is controlled by
$$
\sum_{k=k_0}^\infty \|P_{2^{3k/4}} \partial_r (\psi_1 v) \|_{L_{[t_k,t_{k+1}]}^\infty L_r^2}^2\,,
$$
the square root of which is bounded by $\epsilon^5$ (by Lemma \ref{L:extra-high-freq}).

For the nonlinear term $F_2$, by writing $1=1-\tilde \psi^4 + \tilde \psi^4$, we have
\begin{align*}
F_2
&= - P D_r^s (r^{-4}(1-\tilde \psi^4)\psi_1 |v|^4v)-P D_r^s(r^{-4}\tilde \psi^4\psi_1 |v|^4v)\\
&= F_{21} + F_{22} \,.
\end{align*}
Note that the support of $(1-\tilde \psi^4)\psi_1$ is contained in $|r|\leq \frac12$, and we can use the bootstrap hypothesis \eqref{BSI-3} to obtain
$$
\|F_{21}\|_{L_{[0,T)}^1L_r^2} \lesssim \epsilon^5 \,,
$$
as was done in the proof of Lemma \ref{L:iterator-ring} (for any $s\leq 1$).  For $F_{22}$,  taking $\tilde v = \psi_2v$ and noting that $\psi_1\psi_2=\psi_1$, we have $F_{22} = P D_r^s (r^{-4}\tilde \psi^4\psi_1 |\tilde v|^4\tilde v)$.
By \eqref{E:exchange} with $\alpha=\frac18$,
$$
\|F_{22} \|_{L_{[0,T)}^1 L_r^2} \leq \left\| (T-t)^{-3/32} \| D_r^{s-\frac18} (r^{-4} \tilde \psi^4 \psi_1 |\tilde v|^4\tilde v)  \|_{L_r^2} \right\|_{L_{[0,T)}^1} \,.
$$
Since $\tilde \psi$ is supported in $\frac14\leq |r|\leq r_2$, the function $\tilde \psi^4\psi_1 r^{-4}$ is smooth and compactly supported.  By the fractional Leibniz rule,
\begin{align*}
\| D_r^{s-\frac18} (r^{-4} \tilde \psi^4 \psi_1 |\tilde v|^4\tilde v) \|_{L_r^2}
&\lesssim \|\tilde v\|_{L_r^\infty}^4 \|\la D_r \ra^{s-\frac18} \tilde v \|_{L_r^2} \\
&\lesssim \|D_r^{3/7} \tilde v\|_{L_r^2}^{7/2}\, \|\partial_r \tilde v \|_{L_r^2}^{1/2} \, \|\la D_r \ra^{s-\frac18} \tilde v \|_{L_r^2}.
\end{align*}
Using the bound $\|\partial_r \tilde v\|_{L_r^2} \leq (T-t)^{-1/2}$ from \eqref{E:bs-error-bd} and the bound on $\|D_r^{3/7}\tilde v\|_{L_{[0,T)}^\infty L_r^2}$ from Lemma \ref{L:a-little}, we obtain
\begin{align*}
\| F_{22} \|_{L_{[0,T)}^1 L_r^2}
&\lesssim \| (T-t)^{-3/32} (T-t)^{-1/4}\|_{L_{[0,T)}^1} \|\la D_r\ra^{s-\frac18} \tilde v\|_{L_{[0,T)}^\infty L_r^2} \\
&\lesssim \epsilon^5 \|\la D_r\ra^{s-\frac18} \tilde v\|_{L_{[0,T)}^\infty L_r^2} \,.
\end{align*}

To bound $F_3$, we use \eqref{E:exchange} with $\alpha=\frac98$ to obtain
$$
\|F_3 \|_{L_{[0,T)}^1 L_r^2} \lesssim \| (T-t)^{-27/32} \|_{L_{[0,T)}^1} \| D_r^{s-\frac18} \tilde v \|_{L_{[0,T)}^\infty L_r^2} \,.
$$
The $F_4$ term is more straightforward than $F_3$, since there is one fewer derivative.
\end{proof}

Finally, we can obtain the $H^1$ control, which completes part of the bootstrap estimate \eqref{BSO} in Prop. \ref{P:bootstrap}.

\begin{proposition}[$H^1$ control]
\label{P:ring-energy}
Suppose that the assumptions of Prop. \ref{P:bootstrap} and Remark \ref{R:bs-notation} hold.  Then
$$
\| \partial_r v \|_{L_{[0,T)}^\infty L_{|r|\leq \frac58}^2} \lesssim  \epsilon^5 \,.
$$
\end{proposition}
\begin{proof}
Let $r_k = \frac58+\frac{1}{64}(k-1)$.  Apply Lemma \ref{L:low-freq-iterator} on $[r_k,r_{k+1}]$ for $k=1, \ldots, 8$ to obtain collectively that
$$
\|\partial_r v \|_{L_{[0,T)}^\infty L_{|r|\leq \frac58}^2} \lesssim \epsilon^5 + \|v \|_{L_{|r|\leq \frac34}^2}\leq \epsilon^5
$$
by Lemma \ref{L:a-little}.
\end{proof}

\begin{proposition}[local smoothing control]
\label{P:ring-loc-smooth}
Suppose that the assumptions of Prop. \ref{P:bootstrap} and Remark \ref{R:bs-notation} hold.  Let $\psi_{9/16}$ be a smooth function such that $\psi_{9/16}(r)=1$ for $|r|\leq \frac9{16}$ and $\psi_{9/16}(r)=0$ for $|r|\geq \frac58$.
Then
$$
\| D_r^{3/2} (\psi_{9/16} v) \|_{L_{[0,T)}^2 L_r^2} \lesssim  \epsilon^5 \,.
$$
\end{proposition}
\begin{proof}
Let $\chi(\xi)=1$ for $|\xi|\leq 1$ and $\chi(\xi)= 0$ for $|\xi|\geq 2$ be a smooth function. Let $\chi_-=\chi$ and $\chi_+ = 1-\chi$.  Let $P_-$ be the Fourier multiplier with symbol $\chi_-((T-t)^{3/4}\xi)$ and $P_+$ be the Fourier multiplier with symbol $\chi_+((T-t)^{3/4}\xi)$.  Then $I = P_-+P_+$ for each $t$, and $P_-$ projects onto frequencies $\lesssim (T-t)^{-3/4}$ while $P_+$ projects onto frequencies $\gtrsim (T-t)^{-3/4}$.  Letting $Q$ be the Fourier multiplier with symbol $\frac34 \chi'((T-t)^{3/4}\xi)$, we have $\partial_t P_{\pm} f = \pm i(T-t)^{-1/4}  Q \partial_r f + P\partial_t f$.  Note that $Q$ has Fourier support in $|\xi|\sim (T-t)^{-3/4}$.

First, we can discard low frequencies.  From Prop. \ref{P:ring-energy} and \eqref{E:exchange} with $\alpha=\frac12$,
\begin{align*}
\| D_r^{3/2} P_- \psi_{9/16} v \|_{L_{[0,T)}^2 L_r^2}
&\lesssim \| (T-t)^{-3/8} \partial_r \psi_{9/16} v \|_{L_{[0,T)}^2L_r^2} \\
&\lesssim T^{1/8} \|\partial_r \psi_{9/16} v \|_{L_{[0,T)}^\infty L_r^2} \\
&\lesssim \epsilon^5.
\end{align*}
For the high-frequency portion, $D_r^{3/2} P_+ \psi_{9/16} v$, we first need to dispose of the spatial cutoff.  We have
$$
D_r^{3/2} P_+ \psi_{9/16} = \psi_{9/16} D_r^{3/2} P_+ + [D_r^{3/2} P_+, \psi_{9/16}].
$$
By the pseudodifferential calculus, the leading order term in the symbol of the commutator $[D_r^{3/2} P_+, \psi_{9/16}]$ is $\xi^{1/2}\chi_+(\xi(T-t)^{3/4}) \psi'(r) + \xi^{3/2}(T-t)^{3/4}\chi_+'(\xi(T-t)^{3/4}) \psi'(r)$.
Hence, we obtain the bound
$$
\| [D_r^{3/2} P_+, \psi_{9/16}] \la D_r\ra^{-1/2} \|_{L_r^2\to L_r^2} \lesssim 1 \,,
$$
independently of $t$. Thus, $\|[D_r^{3/2} P_+, \psi_{9/16}]v\|_{L_{[0,T)}^2L_r^2}$ is easily bounded by Prop. \ref{P:ring-energy}.

Consequently, it remains to show that $\|\psi_{9/16} D_r^{3/2}P_+ v\|_{L_{[0,T)}^2L_r^2} \lesssim \epsilon^5$, the estimate for the high-frequency portion with no spatial cutoff to the right of the frequency cut-off.  To obtain local smoothing via the energy method, we need to introduce the pseudodifferential operator $A$ of order $0$ with symbol $\exp(-(\sgn \xi)(\tan^{-1}r))$, where $\sgn \xi$ is a smoothed signum function.  Note that by the sharp G\"arding inequality, $A$ is positive.  The key property of $A$  is
$$
\partial_r^2 Af = A \partial_r^2 f - 2i(1+r^2)^{-1}D_r A f + Bf \,,
$$
where $B$ is an order $0$ pseudodifferential operator.  The first-order term $i(1+r^2)^{-1}D_r A f$ will generate the local smoothing estimate.

Let $w= AP_+ v$.  By the sharp G\"arding inequality,
$$
\| \psi_{9/16} D_r^{3/2} P_+ v \|_{L_{[0,T)}^2 L_r^2} \lesssim \| (1+r^2)^{-1/2} D_r^{3/2} w \|_{L_{[0,T)}^2L_r^2} \,
$$
and it suffices to prove that $\| (1+r^2)^{-1/2} D_r^{3/2} w \|_{L_{[0,T)}^2L_r^2} \lesssim \epsilon^5$.  The equation satisfied by $w$ is
\begin{align*}
i\partial_t w + \partial_r^2 w + 2i(1+r^2)^{-1} D_r w
&= (T-t)^{-1/4}AQ\partial_r v -AP_+r^{-4}|v|^4v  + Bv \, \\
&= F_1+F_2+F_3 \,,
\end{align*}
where $B$ is a order $0$ operator (satisfying bounds independent of $t$).  By applying $\partial_r$ and pairing this equation with $\partial_r w$ (energy method), we obtain, upon time integration,
\begin{align*}
\indentalign \| \partial_r w\|_{L_{[0,T)}^\infty L_r^2}^2 + \|(1+r^2)^{-1/2}D_r^{3/2} w \|_{L_{[0,T)}^2L_r^2}^2 \\
&\lesssim \int_0^T |\la \partial_r F_1, w \ra| + 10 \|\partial_r F_2\|_{L_{[0,T)}^1L_r^2}^2 + 10 \| \partial_r F_3 \|_{L_{[0,T)}^1L_r^2}^2.
\end{align*}
The $F_3$ term is easily controlled using Prop. \ref{P:ring-energy}.

The $F_1$ term is controlled as in the proof of Lemma \ref{L:low-freq-iterator} (a similar first term).  For the $F_2$ term, let $\psi$ be a smooth function such that $ \psi(r) =1$ for $|r|\leq \frac14$ and $\psi(r)=0$ for $|r|\leq \frac12$.  Writing $1=\psi^5 + (1-\psi^5)$, we have
\begin{align*}
F_2 &= AP_+ \psi^5 r^{-4} |v|^4v + AP_+ (1-\psi^5) r^{-4} |v|^4v \\
&= F_{21}+F_{22}.
\end{align*}
We estimate $\| \partial_r F_{21} \|_{L_{[0,T)}^1L_r^2}$ as we did in the proof of Lemma \ref{L:iterator-ring}. For the term $F_{22}$, take $\psi_+ = (1-\psi^5)r^{-4}$, and note that $\psi_+$ is smooth and well-localized.  Recall that in the proof of Lemma \ref{L:iterator-ring} (see \eqref{E:step-90} and \eqref{E:star-2}), we showed that
$$
\|P_{\geq N} \partial_r \psi_+ |v|^4v \|_{L_{I_k}^1 L_r^2} \lesssim 2^{k(1+\delta)/2}N^\delta \beta(k,\mu N) + N^{-1+\delta} 2^{k\delta} \beta(k,\mu N)^2 + N^{-1} 2^{k(1+\delta)/2} \,.
$$
Furthermore, Prop. \ref{P:high-freq-ring} showed that $\beta (k,2^{3k/4}) \lesssim 1$.  Combining the above, gives
$$
\| P_{\geq 2^{3k/4}} \partial_r \psi_+ |v|^4v \|_{L_{I_k}^1L_r^2} \lesssim 2^{-k/8} \,.
$$
Thus,
\begin{align*}
\| \partial_r F_{22} \|_{L_{[0,T)}^1L_r^2}
&\lesssim \sum_{k=k_0}^\infty \| P_{\geq 2^{3k/4}}\partial_r \psi_+ |v|^4v \|_{L_{I_k}^1L_r^2} \\
&\lesssim \sum_{k=k_0}^\infty \|P_{\geq 2^{3k/4}} \partial_r \psi_+ |v|^4v \|_{L_{I_k}^1L_r^2} \\
&\lesssim 2^{-k_0/8} \\
&\lesssim \epsilon^5 .
\end{align*}
\end{proof}

\begin{proposition}[Strichartz control]
\label{P:ring-Strichartz}
Suppose that the assumptions of Prop. \ref{P:bootstrap} and Remark \ref{R:bs-notation} hold.  Then
$$
\|r^{\frac2p-1} \partial_r v \|_{L_{[0,T)}^qL_{|r|\leq \frac12}^p}  \lesssim  \epsilon^5 \,.
$$
\end{proposition}
\begin{proof}
Let $\psi$ be a smooth function such that $\psi(r)=1$ for $|r|\leq \frac12$ and $\psi(r)=0$ for $|r|\geq \frac9{16}$.  Let $w=\psi v$.  Then $w$ solves
\begin{align*}
i\partial_t w + \partial_r^2 w
&= -\psi r^{-4}|v|^4v +2\partial_r (\psi' v) - \psi'' v\\
&= F_1+F_2+F_3.
\end{align*}
By the Strichartz estimate and dual local smoothing estimate, we obtain
$$
\| r^{\frac2{p}-1}\partial_r w \|_{L_{[0,T)}^q L_r^p} \lesssim \| \partial_r w_0 \|_{L_r^2} + \|\partial_r F_1 \|_{L_{[0,T)}^1L_r^2} + \|D_r^{-1/2} \partial_r F_2 \|_{L_{[0,T)}^2L_r^2} + \|\partial_r F_3 \|_{L_{[0,T)}^1L_r^2} \,.
$$
Let $\tilde \psi$ be a smooth function such that $\tilde \psi(r)=1$ for $|r|\leq \frac14$ and $\tilde \psi(r)=0$ for $|r|\geq \frac12$.  By writing $1=\tilde \psi^5 + (1-\tilde \psi^5)$, we have
$$
F_1 = -\psi\tilde \psi^5 r^{-4} |v|^4v -\psi(1-\tilde \psi^5) r^{-4} |v|^4v = F_{11}+F_{12}.
$$
Since the support of $\psi\tilde \psi^5$ is contained in $|r|\leq \frac12$, the term $\| \partial_r F_{11}\|_{L_{[0,T)}^1L_r^2}$ can be estimated by $\epsilon^5$ using bootstrap assumption \eqref{BSI-3} as in the proof of Lemma \ref{L:iterator-ring}.  Since $(1-\tilde \psi^5)\psi r^{-4}$ is a bounded and smooth function,
$$
\| \partial_r F_{12} \|_{L_{[0,T)}^1L_r^2} \lesssim \|\la \partial_r \ra v^5 \|_{L_{[0,T)}^1L_{|r|\leq \frac58}^2} \lesssim T\| \la \partial_r \ra v \|_{L_{[0,T)}^\infty L_{|r|\leq \frac58}^2}^5 \lesssim  \epsilon^5 \,.
$$
Also, by Prop. \ref{P:ring-loc-smooth},
$$
\|D_r^{1/2} F_2 \|_{L_{[0,T)}^2L_r^2} \lesssim \| \la D_r\ra ^{3/2} \psi_{9/16} v \|_{L_{[0,T)}^2L_r^2} \lesssim \epsilon^5 \,.
$$
And finally,
$$
\|\partial_r F_3 \|_{L_{[0,T)}^1L_r^2} \lesssim T \| \la \partial_r \ra v\|_{L_{[0,T)}^\infty L_{|r|\leq \frac58}^2} \lesssim \epsilon^5
$$
by Prop. \ref{P:ring-energy}.  Collecting the above estimates, we obtain the claimed bound.
\end{proof}

This completes the proof of Prop. \ref{P:bootstrap} (via Lemma \ref{L:conversion}).

\end{document}